\documentclass{article}

\usepackage[utf8]{inputenc}

\usepackage{lmodern}       
\usepackage{babel}

\usepackage{indentfirst}
\usepackage{amssymb}
\usepackage{amsmath}
\usepackage{amsfonts}
\usepackage{graphicx}
\usepackage{amsthm}
\usepackage{tipa}
\usepackage{mathrsfs}
\usepackage{enumitem}
\usepackage{xcolor}

\title{Well-posedness of a stochastic Navier–Stokes system with dynamically coupled subgrid scales}
\date{}
\author{Debussche Arnaud$^{2,3}$, Mémin Etienne$^{1}$, Moskowitz Sébastien$^{1}$\\[0.5em]
\small $^{1}$Univ Rennes, Inria, Odyssey, IRMAR-UMR 6625, Centre Inria Rennes 35042 Rennes Cedex, France\\
\small $^{2}$Univ Rennes, CNRS, IRMAR-UMR 6625, F-35000 Rennes, France\\
\small $^{3}$Institut universitaire de France (IUF)}

\newcommand{\R}[1]{\mathbb{R}^{#1}}
\newcommand{\N}{\mathbb{N}}

\newcommand{\normeH}[1]{\left\vert #1\right\vert_H}
\newcommand{\norme}[2]{\left\Vert #1\right\Vert_{#2}}
\newcommand{\scalar}[4][\big]{#1(#2,#3#1)_{#4}}
\newcommand{\bsym}[1]{\boldsymbol{#1}}
\newcommand{\domA}[1]{\mathcal{D}(A^{#1})}
\newcommand{\esp}[2][\big]{\mathbb{E}#1[#2\mathclose{#1]}}

\newcommand{\F}[3][]{F^{#1}\big[{#2}\big]({#3})}
\newcommand{\G}[3][]{G^{#1}\big[{#2}\big]({#3})}
\newcommand{\bcdot}{\bsym{\cdot}}

\newcommand{\rev}[1]{#1}

\newtheorem{lemme}{Lemma}
\newtheorem{theorem}{Theorem}
\newtheorem{definition}{Definition}

\begin{document}

\maketitle

\begin{abstract}
We study a stochastic Navier--Stokes system in which the spatial modes of the transport noise evolve dynamically and are coupled to the resolved velocity. The model couples a Navier--Stokes stochastic PDE to an infinite family of linearized Navier-Stokes equations for the noise correlation modes. These latter equations may contain a hyperviscosity. We prove existence of martingale solutions in two dimensions for hyperviscosity exponent $s\geq1$, and pathwise uniqueness for $s>1$ under additional regularity of the initial noise modes. In three dimensions, we establish existence of martingale solutions for $s>3/2$. The proof combines Galerkin approximations and compactness with estimates adapted to the coupled drift and stochastic terms.
\end{abstract}

\section{Introduction}\label{sec:introduction}

Resolving interactions across the full range of spatial and temporal scales in turbulent geophysical flows remains well beyond the capabilities of present and near-future numerical methods. The range of active scales is immense, extending from solar-driven motions on the order of 10,000 km down to dissipative processes at the millimeter scale. Within this vast cascade, subgrid processes play a central role. They must capture the spread of possible realizations arising from uncertain, coarse initial data, represent energy transfers across scales, and ensure the overall stability of the simulation. Moreover, the coexistence of multiple physical processes — such as waves and turbulence — together with intermittent forcing further complicates these interactions. Probabilistic modeling offers a natural and promising framework for addressing these challenges while providing natural ensemble forecasting capabilities.\\

Among the probabilistic frameworks proposed in recent years \cite{Berner_et_al, Gottwald_et_al}, transport- or advection-based approaches have established themselves within the mathematical and physical communities \cite{Agresti-et-al-2022, Debussche_Memin, Memin04032014, Flandoli-Pappalettera-2023, Lang2023}. These models exhibit appealing mathematical properties \cite{Flandoli-Galeati-Luo-2021, Street-Crisan-23, Goodair2024} and enable the systematic derivation and justification of stochastic representations grounded in physical conservation laws \cite{Agresti-et-al-2022, Li2025, Street-Crisan-23}. These frameworks have been implemented and tested across a wide range of configurations, from idealized models to realistic numerical codes \cite{Bauer_Chandramouli_Li_Memin, Brecht_Li_Bauer_Memin, Cotter_et_al_2019, Chapron_Derian_Memin_Resseguier, Resseguier_Memin_Chapron, Resseguier_Memin_Chapron_2}. They have been shown to represent several physical mechanisms such as waves \cite{Li_Lahaye_Memin_Deremble, Memin_Chapron_Marie, Li_Memin_Tissot} and deep convection \cite{Jamet_Memin_Dumas_Li_Garreau}, and to extend to principled compressible settings  \cite{Tissot_Memin_Jamet}.
Most of these frameworks nevertheless assume that the noise is prescribed, either from data or through a specified functional form that may depend on the resolved large-scale flow. However, no explicit dynamics is imposed on the small-scale stochastic component. In all these models based on stochastic partial differential equations (SPDEs), the noise is typically represented as a spatially correlated, temporally white process. Its structure is prescribed a priori and, from a mathematical standpoint, often requires strong assumptions on the regularity of the noise basis to ensure well-posedness.
 
In this work, we focus on an extension of the modeling under Location Uncertainty (LU) framework \cite{Memin04032014}, a modeling setting based on transport noise derived from the Reynolds transport theorem. This extension, introduced in \cite{Debussche_Memin}, is obtained from two coupled variational principles: a pathwise principle governing the large-scale dynamics, together with a second principle formulated in expectation. This latter principle provides evolution equations for the noise correlation basis functions and thereby introduces a genuine dynamics for the small scales, which is coupled to the SPDE governing the large-scale components.
For the Euler equations, the resulting small-scale dynamics takes the form of a linearized Euler-type equation involving transport and stretching by the large-scale flow, while the large-scale component evolves according to an LU-type SPDE. It is important to emphasize that the small-scale linear model does not arise from a classical linearization procedure, but directly from the variational principle in expectation. Finally, the small-scale dynamics can be interpreted as a generalization of the classical Kraichnan–Tennekes \cite{Kraichnan_1961} turbulence model, in which small scales are advected by a large-scale random velocity field.

The well-posedness of the LU equations has been studied in \cite{Debussche_Hug_Memin_2023}. More broadly, within the class of Navier–Stokes-type equations, establishing well-posedness remains a central issue in mathematical fluid dynamics. Notably, the question of global regularity for the three-dimensional Navier–Stokes equations remains open to this day.
For the deterministic Navier-Stokes equations, the existence of weak solutions was established by Leray in the 1930s, see \cite{Leray1934}. In contrast, stochastic Navier–Stokes equations with additive noise have been shown to be well-posed {since the 1970s,} see \cite{BENSOUSSAN1973195}. In particular, the work of F. Flandoli and D. Gatarek \cite{FlandoliGatarek1995} provides a foundational result on existence  for stochastic Navier–Stokes equations, adapting techniques from the deterministic theory to the stochastic framework. The next step in the study of stochastic Navier-Stokes equation is to consider transport-related noise. The well-posedness results were adapted to this equation in \cite{Mikulevicius_Rozovskii_2005}.

{The purpose of this paper is to establish well-posedness for this coupled system. Unlike the prescribed-noise setting of \cite{Debussche_Hug_Memin_2023}, the coefficients of the stochastic transport operator and of its It\^o correction are themselves unknowns. Thus compactness must be obtained simultaneously for the velocity and for the infinite family of evolving noise modes. Compared with \cite{Debussche_Memin}, we consider the viscous Navier--Stokes system and address its stochastic well-posedness, introducing hyperviscosity in the noise-mode equations when additional regularity is needed.}

{In two dimensions we construct martingale solutions for $s\geq1$ and prove pathwise uniqueness for $s>1$ under additional regularity of the initial noise modes. In three dimensions we obtain martingale solutions for $s>3/2$. The proofs use Galerkin approximation, compactness and a coupled stability estimate.}

A fundamental difficulty of the system studied here, compared with the standard LU equations with prescribed noise, is its quasilinear character. Indeed, the diffusion term in the large-scale Navier--Stokes equation involves products of the form

$$
\sum_{i\geq 0}\nabla\cdot\bigl(\xi_i\xi_i^\top\nabla u\bigr),
$$
where both the large-scale velocity \(u\) and the noise modes \((\xi_i)_{i\geq0}\) are unknowns. Thus, the effective diffusion operator itself evolves with the solution. This substantially complicates the compactness and stability analysis and motivates the introduction of hyperviscosity in the evolution equations for the noise modes.

\section{A coupled multiscale system}\label{sec:a_coupled_system}

In this section, we present the coupled system and all the terms involved. We also recall the notions of martingale and pathwise solutions.

\subsection{Mathematical framework}
We consider a smooth bounded domain $\mathcal{S} \subset \R{d}.$ Let $\bigl(\Omega, \mathcal{F}, \{\mathcal{F}_t\}_{t \geq 0}, \mathbb{P}\bigr)$ be a filtered probability space, composed of a measurable space $(\Omega, \mathcal{F})$, a filtration $\{\mathcal{F}_t\}_{t \geq 0}$ on $\mathcal{F}$, and a probability measure $\mathbb{P}.$ We define on this filtration a sequence of independent real-valued Brownian motions $(\beta_i)_{i \in \N}.$ In this article, we study the following system of coupled stochastic equations (see \cite{Debussche_Memin} for a complete derivation of the system). The unknowns are the fluid velocity and pressure $(\bsym{u}, p)$ for the first equation, and the sequence of noise correlation functions together with a pressure term $\bigl(\bsym{\xi} = (\bsym{\xi}_i)_{i \in \N}, (q_i)_{i \in \N}\bigr)$ for the second equation:
\begin{equation}
	\label{eq:systeme}
	\left\{
	\begin{aligned}
		&d_t \bsym{u} - \nu_1 \Delta \bsym{u} + \Bigl[(\bsym{u}  \bcdot \nabla) \bsym{u} - \nabla \bcdot \sum_{i=0}^{\infty} \bsym{\xi}_i \bsym{\xi}_i^{\intercal} \nabla\bsym{u} \Bigr]dt \\
        &\qquad\qquad\qquad\qquad\qquad= -\sum_{i=0}^{\infty} (\bsym{\xi}_i d\beta_i \bcdot \nabla) \bsym{u}  - \nabla (dp- dp_{\sigma}),  \\
		&\partial_t \bsym{\xi}_i - \nu_2 \Delta^s \bsym{\xi}_i + (\bsym{\xi}_i \bcdot \nabla)\bsym{u} + (\bsym{u} \bcdot \nabla) \bsym{\xi}_i = -\nabla q_i, \qquad i \in \mathbb{N}, \\[0.5em]
        &\nabla \bcdot \bsym{u} = 0, \quad \nabla \bcdot \bsym{\xi}_i = 0 , \qquad i \in \N.
	\end{aligned}
	\right.
\end{equation}
Here, $s \geq 1.$  The unknowns are random processes depending on time and space: $t \in [0,T]$ and $\bsym{x} \in \mathcal{S}.$ Both $\nu_1, \nu_2 > 0$ are two positive real numbers. We now define the Leray projection as the orthogonal projection on the space $H$ of square integrable functions with $0$ divergence (a formal definition is given below):
\begin{equation*}
P : L^2(\mathcal{S}, \R{d}) \longrightarrow H.
\end{equation*}
We apply the Leray projector on equation \eqref{eq:systeme}. This enables us to get rid of the pressure terms. We obtain the following coupled system with unknown $(\bsym{u}, \bsym{\xi}),$ where $\bsym{\xi} = (\bsym{\xi}_i)_{i \in \N}$ is a sequence of unknown functions:
\begin{equation}
	\label{eq:systeme_projecteur_Leray}
	\left\{
	\begin{aligned}
		d_t \bsym{u} - \nu_1 P \Delta \bsym{u} &+ P\bigl((\bsym{u} \bcdot \nabla) \bsym{u}\bigr) dt \\
        &- P\Bigl(\nabla \bcdot \sum_{i=0}^{\infty} \bsym{\xi}_i \bsym{\xi}_i^{\intercal} \nabla\bsym{u}\Bigr) dt = -P\Bigl(\sum_{i=0}^{\infty} (\bsym{\xi}_i d\beta_i \bcdot \nabla) \bsym{u} \Bigr), \\[0.5em]
		\partial_t \bsym{\xi}_i - \nu_2 P\Delta^s \bsym{\xi}_i &+ P\bigl((\bsym{\xi}_i \bcdot \nabla)\bsym{u}\bigr) + P\bigl((\bsym{u} \bcdot \nabla) \bsym{\xi}_i\bigr) = 0, \qquad i \in \mathbb{N}.
	\end{aligned}
	\right.
\end{equation}
This system comes with the incompressibility conditions
\begin{equation}
\label{eq:hypothese_incompressible}
\nabla \bcdot \bsym{u}(t,\bsym{x}) = 0, \quad \nabla \bcdot \bsym{\xi}_i(t,\bsym{x})  = 0, \quad \forall i \in \N, t \in [0,T], \bsym{x} \in \mathcal{S},
\end{equation}
the boundary conditions 
\begin{equation}
\label{eq:boundary_condition}
\bsym{u}(t,\bsym{x}) = 0, \quad \bsym{\xi}_i(t,\bsym{x}) = 0, \quad \forall i \in \N, t \in [0,T], \bsym{x} \in \partial \mathcal{S},
\end{equation}
and the initial conditions
\begin{equation}
\label{eq:initial_condition}
\bsym{u}(0,\bsym{x}) = \bsym{u}_{0}(\bsym{x}), \quad \bsym{\xi}_i(0,\bsym{x}) = \bsym{\xi}_{i,0}(\bsym{x}), \quad \forall i \in \N,  \bsym{x} \in \mathcal{S}.
\end{equation}
Let us now write the system in its abstract form. From now on, since these constants play no role in our analysis, we set \(\nu_1=\nu_2=1\). We follow the presentation and notation introduced in \cite{lions_1969}. First, we define the set
\begin{equation*}
\mathcal{V} = \{\phi \in C^{\infty}_c(\mathcal{S})^d, \nabla \bcdot \phi = 0  \}.
\end{equation*} 
Let $H$ be the closure of $\mathcal{V}$ in $L^2(\mathcal{S}, \R{d})$ and $V$ the closure of $\mathcal{V}$ in $H^1(\mathcal{S}, \R{d}).$  The space H is endowed with the $L^2(\mathcal{S}, \R{d})$ scalar product: 
\begin{equation*}
\scalar[\big]{\bsym{u}}{\bsym{v}}{H} = \scalar[\big]{\bsym{u}}{\bsym{v}}{L^2(\mathcal{S}, \R{d})} \quad \text{and} \quad \normeH{\bsym{u}} = \norme{\bsym{u}}{L^2(\mathcal{S}, \R{d})}. 
\end{equation*}
$V$ is endowed with the $H^1_0(\mathcal{S}, \R{d})$ scalar product and norm:
\begin{equation*}
\scalar[\big]{\bsym{u}}{\bsym{v}}{V} = \scalar[\big]{\nabla\bsym{u}}{\nabla\bsym{v}}{L^2(\mathcal{S}, \R{d})} \quad \text{and} \quad \norme{\bsym{u}}{V} = \norme{\nabla\bsym{u}}{L^2(\mathcal{S}, \R{d})}.
\end{equation*}
We denote by $V'$ the dual space of $V.$ The space of Hilbert-Schmidt operators from the Hilbert space $K_1$ to the Hilbert space $K_2$ is denoted $\mathcal{L}_2(K_1,K_2)$ and its norm $\norme{.}{\mathcal{L}_2(K_1,K_2)}.$ \\

We now define operator $A : \domA{} = V \cap H^{2}(\mathcal{S}, \R{d}).$ It is an unbounded, densely defined, bijective operator such that
\begin{equation}
\label{eq:def_operator_A}
A = - P \Delta, \quad \text{and satisfies} \quad \forall \bsym{u}, \bsym{v} \in \domA{} \times V, \scalar[\big]{A\bsym{u}}{\bsym{v}}{H} = \scalar[\big]{\bsym{u}}{\bsym{v}}{V}.
\end{equation}
Since the embedding $V \hookrightarrow H$ is compact, it follows that $A^{-1}$ is compact on $H.$ It is also a symmetric operator. Therefore, there exists a complete orthonormal basis $(\bsym{e}_i)_{i\geq 0}$ for $H$ composed of eigenfunctions of $A,$ each associated to the increasing unbounded sequence of eigenvalues $(\lambda_i)_{i \geq 0}.$ We now define the fractional powers of $A.$ Consider $\alpha > 0,$ and take
\begin{equation}
\domA{\alpha} = \Bigl\{ \bsym{u} \in H ; \sum_{i=0}^{\infty} \lambda_i^{2\alpha} \scalar[\big]{\bsym{u}}{\bsym{e}_i}{H}^2 < \infty \Bigr\}.
\end{equation}
On $\domA{\alpha},$ we define operator $A^{\alpha}$ 
\begin{equation}
A^{\alpha}\bsym{u} = \sum_{i = 0}^{\infty} \lambda_i^{\alpha} \scalar[\big]{\bsym{u}}{\bsym{e}_i}{H}\bsym{e}_i.
\end{equation}
The domain $\domA{\alpha}$ can be equipped  with the Hilbert norm 
\begin{equation}
\norme{\bsym{u}}{2\alpha} = \normeH{A^{\alpha}\bsym{u}} = \Bigl(\sum_{i=0}^{\infty} \lambda_i^{2\alpha}\scalar[\big]{\bsym{u}}{\bsym{e}_i}{H}^2\Bigr)^{1/2}.
\end{equation}
It is well known that $\norme{.}{2\alpha}$ is equivalent to  the norm of the space $H^{2\alpha}(\mathcal{S},\R{d}).$ Operator $B : V\times V \longrightarrow V'$ and operator $b : V\times V \times V \longrightarrow \R{}$ are defined respectively by
\begin{equation}
\label{eq:def_operator_B}
B(\bsym{u},\bsym{v}) = P \big[(\bsym{u} \bcdot \nabla )\bsym{v}\big] \quad \text{and} \quad b(\bsym{u}, \bsym{v}, \bsym{w}) = \scalar[\big]{B(\bsym{u}, \bsym{v})}{\bsym{w}}{H}.
\end{equation}
Note that, thanks to the incompressibility conditions we have
\begin{equation}
\label{eq:b_identity}
b(\bsym{u}, \bsym{v}, \bsym{w}) = -b(\bsym{u}, \bsym{w}, \bsym{v}) \quad \text{and} \quad b(\bsym{u}, \bsym{v}, \bsym{v}) = 0 \quad \forall \bsym{u}, \bsym{v}, \bsym{w} \in V. 
\end{equation}
For any $\beta > \frac{d}{2}, \domA{\beta/2} \hookrightarrow L^{\infty}(\mathcal{S},\R{d}),$ and consider
\begin{equation*}
\gamma > \frac{d+2}{4}.
\end{equation*}
As such, both operators $A$ and $B$ can be extended to continuous operators:
\begin{equation*}
A : H \longrightarrow \domA{-\gamma} \quad B : H \times H \longrightarrow \domA{-\gamma}.
\end{equation*}
We have the useful upper bounds:
\begin{equation}
\label{eq:A-B_majoration}
\norme{A\bsym{v}}{\domA{-\gamma}} \leq \norme{\bsym{v}}{V} \quad \text{and} \quad \norme{B(\bsym{v}, \bsym{w})}{\domA{-\gamma}} \leq \normeH{\bsym{v}}\normeH{\bsym{w}}.
\end{equation}
Consider a sequence of $V-$valued functions $(\bsym{\xi}_i)_{i \in \N}$ defined on $[0,T] \times \mathcal{S}.$ We now define the stochastic operators $G[(\bsym{\xi}_i)_{i \in \N}](.) : V \longrightarrow \mathcal{L}_2\bigl(L^2(\mathcal{S}, \mathbb{R}^d), H\bigr)$ and the associated Itô correction $F[(\bsym{\xi}_i)_{i \in \N}] (.) : V \longrightarrow H.$ Consider $\phi = \sum_{i = 0}^{\infty}  \phi_i \bsym{e}_i \in H :$
\begin{equation}
\label{eq:def_operators_2}
\begin{aligned}    
G[(\bsym{\xi}_i)_{i \in \N}] (\bsym{u})\phi &= - P\Bigl(\sum_{i=0}^{\infty} (\phi_i\, \bsym{\xi}_i \bcdot \nabla) \bsym{u}\Bigr)  \\ 
F[(\bsym{\xi}_i)_{i \in \N}] (\bsym{u}) &= -\frac{1}{2}P  \Bigl( \nabla \bcdot\sum^{\infty}_{i=0} \bsym{\xi}_i \bsym{\xi}^{\intercal}_i \nabla \bsym{u}\Bigr).
\end{aligned}
\end{equation}
Henceforth, to simplify notations we drop the $i \in \mathbb{N}$ index and we denote by $\bsym{\xi}$ the sequence $(\bsym{\xi}_i)_{i \in \N}$. Once again both operators $\F{\bsym{\xi}}{\bsym{u}}$ and $\G{\bsym{\xi}}{\bsym{u}}$ can be extended to continuous operators
\begin{equation*}
\F{\bsym{\xi}}{\bsym{.}} : H \longrightarrow \mathcal{D}(A^{-\gamma}),
\qquad
\G{\bsym{\xi}}{\bsym{.}} : H \longrightarrow \mathcal{L}_2\bigl(L^2(\mathcal{S}, \mathbb{R}^d), \domA{-\gamma}\bigr).
\end{equation*}
Here we give an upper bound to $\norme{\F{\bsym{\xi}}{\bsym{\bsym{u}}}}{\domA{-\gamma}}.$ For $\bsym{u}, \bsym{\xi}  \in H\times H^\N,$ 
\begin{align*}
\norme{\F{\bsym{\xi}}{\bsym{\bsym{u}}}}{\domA{-\gamma}} &\leq \sum_{i=0}^{\infty} \norme{\bsym{\xi}_i\bsym{\xi}_i^{\intercal} \nabla  \bsym{u}}{\domA{-\gamma + \frac{1}{2}}} \\
& \leq \sum_{i=0}^{\infty} \norme{\bsym{\xi}_i\bsym{\xi}_i^{\intercal}\nabla\bsym{u}}{L^1}. 
\end{align*}
Thanks to Cauchy-Schwarz inequality we find an upper bound for operator $F:$
\begin{equation}
\label{eq:F_majoration}
\norme{\F{\bsym{\xi}}{\bsym{\bsym{u}}}}{\domA{-\gamma}}  \leq \norme{\bsym{u}}{V} \sum_{i=0}^{\infty} \norme{\bsym{\xi}_i}{L^4}^2 .
\end{equation}
We now compute $\norme{\G{\bsym{\xi}}{\bsym{u}}}{\mathcal{L}_2\bigl(L^2(\mathcal{S}, \mathbb{R}^d), \domA{-\gamma}\bigr)}^2$ depending on the quantities $\norme{\bsym{\xi}_i}{H}$ and $\norme{\bsym{u}}{H}$ for $\bsym{\xi} \in V^\N$ and $\bsym{u} \in V:$
\begin{align*}
\norme{\G{\bsym{\xi}}{\bsym{u}}}{\mathcal{L}_2\big(L^2(\mathcal{S}, \mathbb{R}^d), \domA{-\gamma}\big)}^2 &= \sum_{i=0}^{\infty} \norme{\G{\bsym{\xi}}{\bsym{u}}\bsym{e}_i}{\domA{-\gamma}}^2 \\
&= \sum_{i=0}^{\infty} \norme{B(\bsym{\xi}_i, \bsym{u})}{\domA{-\gamma}}^2.  
\end{align*}
This gives us the upper bound for operator $G$ thanks to \eqref{eq:A-B_majoration}: 
\begin{equation}
\label{eq:G_majoration}
\norme{\G{\bsym{\xi}}{\bsym{u}}}{\mathcal{L}_2\big(L^2(\mathcal{S}, \mathbb{R}^d), \domA{-\gamma}\big)}^2 \leq \sum_{i=0}^{\infty} \normeH{\bsym{\xi}_i}^2 \normeH{\bsym{u}}^2.
\end{equation}
Since $\G{\bsym{\xi}}{\bsym{u}}$ is Hilbert-Schmidt, the following stochastic term is well defined and must be understood in the sense of Itô integration:
\begin{equation}
\G{\bsym{\xi}}{\bsym{u}}d\bsym{W}_t =  - P\Bigl(\sum_{i=0}^{\infty} (d\beta_i\, \bsym{\xi}_i \bcdot \nabla) \bsym{u}\Bigr).  
\end{equation}
As such, the process $\bsym{W}_t = \sum_{i=0}^{\infty} \beta_i \bsym{e}_i $ defines a cylindrical Wiener process on $H,$ it takes values in $\domA{-\gamma}.$ We can now write the abstract formulation of the system for $s \geq 1$ on the stochastic basis $\bigl(\Omega, \mathcal{F}, \{\mathcal{F}_t\}_{t \in \mathbb{R}_+}, \mathbb{P}, (\beta_i)_{i \in \N}\bigr):$
\begin{equation}
	\label{eq:equation_operateurs}
	\left\{
	\begin{aligned}
&d_t \bsym{u} + \bigl( A\bsym{u} + B(\bsym{u},\bsym{u}) + \F{\bsym{\xi}}{\bsym{u}}\bigr) dt = \G{\bsym{\xi}}{\bsym{u}}d\bsym{W}_t, \\[0.5em]    
&\partial_t \bsym{\xi}_i +  A^s \bsym{\xi}_i + B(\bsym{\xi}_i,\bsym{u}) + B(\bsym{u},\bsym{\xi}_i) = 0, \qquad i \in \N. 
	\end{aligned}
	\right.
\end{equation}
{Had we replaced the noise-basis evolution equation} by a stationary equation $\partial_t \bsym{\xi}_i = 0, i \in \N,$ we would have recovered the LU system and we refer to \cite{Debussche_Hug_Memin_2023} for the study of the well-posedness of the equation. From now on, we will denote by $L^p$ the space $L^p(\mathcal{S},\R{d}),$ where $p \geq 1.$ Similarly, we will write $H^{\alpha}$ instead of $H^{\alpha}(\mathcal{S},\R{d}).$

\subsection{Solutions and main results}

We give the definitions of martingale and pathwise solutions. Let  $\alpha \in [0,1)$ and $q<\alpha$. We recall that $\gamma > \frac{d+2}{4},$ where $d$ stands for the space dimension.
\begin{definition}[Martingale Solution]
\label{def:martingale_solution}
We say that there exists a martingale solution of equation \eqref{eq:equation_operateurs} if there exists a stochastic basis $\bigl(\Omega, \mathcal{F}, \{\mathcal{F}_t \}_{t\in [0,T]}, \mathbb{P}, \bsym{W}_t \bigr),$ a pair of progressively measurable processes \(u\) and \(\xi=(\xi_i)_{i\in\mathbb N}\), where \(u\) is \(H\)-valued and each \(\xi_i\) is \(H\)-valued, such that, \(\mathbb P\)-a.s.,
\begin{align*}
	&\bsym{u}(.,\omega) \in L^2\bigl(0,T,V\bigr) \cap L^{\infty}\bigl(0,T,H\bigr) \cap C^0\bigl(0,T,\domA{-\gamma}\bigr),  \\
	&\bigl(\bsym{\xi}_i(.,\omega)\bigr)_{i \in \N} \in l^2\Bigl(\N, L^2(0,T,\domA{\frac{1+\alpha}{2}}) \cap L^{\infty}(0,T,\domA{\frac{\alpha}{2}}) \cap C^0\bigl(0,T,\domA{\frac{q}{2}}\bigr)\Bigl)  ,
\end{align*}
and such that $\mathbb{P}-a.s.$ the identity
\begin{equation}
\label{eq:formulation_variationnelle_operator}
\left\{
\begin{aligned}
\scalar[\big]{\bsym{u}(t) - \bsym{u}(0)}{\bsym{y}}{H} + \displaystyle\int_{0}^t \scalar[\big]{A\bsym{u}}{\bsym{y}}{H} + \scalar[\big]{B(\bsym{u},\bsym{u})}{\bsym{y}}{H} \\
+ \scalar[\big]{\F{\bsym{\xi}}{\bsym{u}}}{\bsym{y}}{H}  ds = \scalar[\Big]{\displaystyle\int_0^t \G{\bsym{\xi}}{\bsym{u}}d\bsym{W}_s}{\bsym{y}}{H}, \\[0.5em]
\scalar{\bsym{\xi}_i(t) - \bsym{\xi}_i(0)}{\bsym{z}}{H} + \displaystyle \int_{0}^t \scalar[\big]{A^s \bsym{\xi}_i}{\bsym{z}}{H} + \scalar[\big]{B(\bsym{\xi}_i,\bsym{u})}{\bsym{z}}{H} \\
\quad + \scalar[\big]{B(\bsym{u},\bsym{\xi}_i)}{\bsym{z}}{H} ds = 0, \qquad i \in \N.
\end{aligned}
\right.
\end{equation}
holds for all $t \in [0,T]$ and all $\bsym{y}, \bsym{z} \in \domA{\gamma}.$
\end{definition}

We define the second type of stochastic solutions. 

\begin{definition}[Pathwise solution]
\label{def:pathwise_solution}
Consider $\mathfrak{S} = \bigl(\Omega, \mathcal{F}, \{\mathcal{F}_t\}_{t \in [0,T]}, \mathbb{P}, \bsym{W}_t\bigr)$ a fixed stochastic basis. A pathwise solution of equation \eqref{eq:equation_operateurs} with respect to $\mathfrak{S}$ is a pair of progressively measurable stochastic processes $(\bsym{u},\bsym{\xi})$ such that $(\mathfrak{S}, \bsym{u}, \bsym{\xi})$ is a martingale solution of equation \eqref{eq:equation_operateurs}.
\end{definition}

{Here ``pathwise solution'' means a solution constructed on a prescribed stochastic basis and driven by the prescribed Wiener process. Pathwise uniqueness means that two such solutions on the same stochastic basis, with the same Wiener process and the same initial data, are indistinguishable.}

 The first result of this work is a well-posedness theorem in dimension $d=2:$

\begin{theorem}[Existence and uniqueness of solution, $d=2$]
	\label{thm:well_posedness_2d} The following statements hold:

	\begin{enumerate}[label=(\roman*)]
		\item \textbf{Existence:}
		Let $d=2$ and $s\geq1.$ For all $\alpha \in [\frac{1}{2},1[$ and any initial condition $\bsym{u}_0\in H$,  $\bsym{\xi}_{0} \in H^{\mathbb N}$, such that $ \sum_{i=0}^{\infty} \norme{\bsym{\xi}_{i,0}}{\alpha}^2 < +\infty,$ there exists a martingale solution to \eqref{eq:equation_operateurs}.
		
		\item \textbf{Pathwise existence and uniqueness:}
		{Let $d=2$ and $s>1$, and assume in addition that there exists $\alpha_*\in(1,s)$ such that
			$\sum_{i=0}^{\infty}\norme{\bsym{\xi}_{i,0}}{\alpha_*}^2<+\infty$. Consider a prescribed stochastic basis $\mathfrak{S}=\bigl(\Omega,\mathcal{F},\{\mathcal{F}_t\}_{t\in[0,T]},\mathbb{P},\bsym{W}_t\bigr)$. Then there exists a pathwise solution to \eqref{eq:equation_operateurs} on $\mathfrak{S}$, and pathwise uniqueness holds; in particular, any two solutions on the same stochastic basis with the same initial data and Wiener process are indistinguishable.}
	\end{enumerate}
\end{theorem}

{The regularity thresholds appearing in Theorem \ref{thm:well_posedness_2d} have different analytical origins. The condition $s\geq1$ is sufficient for the a priori estimates and compactness argument used to construct martingale solutions. The strict condition $s>1$ in the uniqueness statement is used to choose an exponent $\alpha_*\in(1,s)$; the embedding $H^{\alpha_*}\hookrightarrow L^\infty$ in two dimensions is then crucial in the coupled stability estimate. We do not claim that the threshold $s>1$ is optimal; the endpoint $s=1$ is left open.}

\rev{Section \ref{sec:approximation_scheme} is devoted to the construction of martingale solutions in two dimensions by Galerkin approximation and compactness. Pathwise uniqueness and existence on a prescribed stochastic basis are proved in Section \ref{sec:pathwise_solution}. The three-dimensional case is treated in Section \ref{sec:the3d_case}.} In the next result, we take $\alpha =0.$

\bigskip

\begin{theorem}[Existence of solution, $d=3$]
	\label{thm:well_posedness_3d}   
	Let $d=3$ and $s > 3/2.$ For any initial conditions $\bsym{u}_0\in H$,  $\bsym{\xi}_{0} \in H^{\mathbb N}$, such that $\sum_{i=0}^{\infty} \normeH{\bsym{\xi}_{i,0}}^2 < +\infty,$ there exists a martingale solution to \eqref{eq:equation_operateurs}.
\end{theorem}
The energy estimate for the noise modes remains valid at the endpoint \(s=3/2\). The strict condition \(s>3/2\) arises in the compactness argument for the velocity: it allows us to choose \(p>1\) such that

$$
\frac{3p}{s(2-p)}\leq 2.
$$

This condition is used to obtain the required time-space compactness. We do not address the endpoint \(s=3/2\).


\section{Construction of martingale solutions}\label{sec:approximation_scheme}

\bigskip 
We consider  $s=1$ for the construction of martingale solutions, the proof for $s>1$ is of course easier.  We define the projected operators and introduce a Galerkin scheme.

Recall that  $(\bsym{e}_i)_{i \in \N}$ denotes the sequence of eigenvectors of  $A.$ They form a Hilbertian basis of $H.$ Define by $P_n : \domA{-\gamma} \longrightarrow H_{n} := Span(\bsym{e}_{0}, ... \bsym{e}_{n})$ the canonical orthogonal projection. Recall that $(\lambda_i)_{i \in \N}$ is the increasing unbounded sequence of eigenvalues of  $A.$  We have 

\begin{equation*}
P_{n}\bsym{u} := \sum_{i=0}^{n} \big\langle\bsym{u},\bsym{e}_{i} \big\rangle_{\domA{-\gamma}\times \domA{\gamma}}\bsym{e}_i \qquad \forall \bsym{u} \in \domA{-\gamma}, \gamma \in \mathbb{R}_+.
\end{equation*}
$P_n$ satisfies: 
\begin{equation}
\label{eq:estimation_norme_proj}
\norme{P_{n}\bsym{u}}{\domA{-\gamma}} \leq \norme{\bsym{u}}{\domA{-\gamma}}, \gamma \in \mathbb{R}.
\end{equation}
By dominated convergence, we have: 
\begin{equation}
\label{eq:convergence_norme_proj}
\norme{P_{n}\bsym{u}-\bsym{u}}{\domA{\gamma}} \underset{n \rightarrow \infty}{\longrightarrow} 0, \quad \bsym{u} \in \domA{\gamma}, \gamma \in \mathbb{R}. 
\end{equation}
And also for $\alpha < \beta$ and $n \in \N:$
\begin{equation}
\label{eq:majoration_projecteur}
\begin{aligned}
\norme{P_{n}\bsym{u}}{\domA{\beta}} &\leq \lambda^{\beta-\alpha}_{n+1} \norme{P_{n}\bsym{u}}{\domA{\alpha}}, \\
\norme{(I-P_{n})\bsym{u}}{\domA{\alpha}} &\leq \lambda^{\alpha-\beta}_{n+1} \norme{(I-P_{n})\bsym{u}}{\domA{\beta}}.
\end{aligned}
\end{equation}
We also define the  projected operators. For a given sequence of functions $\bsym{\xi} \in H^\N:$
\begin{equation}
B^{n} := P_{n}B, \quad \F[n]{\bsym{\xi}}{.} := P_{n}\F{P_n\bsym{\xi}}{.}  \quad \text{and} \quad \G[n]{\bsym{\xi}}{.}  := P_{n}\G{P_n\bsym{\xi}}{.} .
\end{equation} 
Consider the projected system 
\begin{equation}
\label{eq:projected_system}
\left\{
	\begin{array}{l}
    d_t \bsym{u}_n + \bigl(A\bsym{u}_n + B^n(\bsym{u}_n,\bsym{u}_n) + \F[n]{\bsym{\xi}^n}{\bsym{u}_n}\bigr) dt = \G[n]{\bsym{\xi}^n}{\bsym{u}_n} d\bsym{W}_t, \\[0.5em]
    \partial_t \bsym{\xi}^n_i + A \bsym{\xi}^n_i + B^n(\bsym{\xi}^n_i,\bsym{u}_n) + B^n(\bsym{u}_n,\bsym{\xi}^n_i) = 0, \qquad i \in \N.
	\end{array}
	\right.
\end{equation}
\rev{Operators $A, B^n, \F[n]{\bsym{\xi}}{.}$ and $\G[n]{\bsym{\xi}}{.}$ are Lipschitz on $H_n$, since $H_n$ is finite-dimensional.} 
Hence, it can be shown easily that for fixed \(n\), the projected system is an ODE/SDE on the complete product space \(H_n\times\ell^2(\mathbb N;H_n)\) with locally Lipschitz coefficients. The standard Cauchy–Lipschitz theorem yields a unique local solution. The energy estimates below prevent blow-up and therefore extend the solution globally.
We denote the  solution by $(\bsym{u}_n,\bsym{\xi}^n)$ on $[0,T]$ for arbitrary $T>0$. From now on, we will fix such a $T >0.$


We now derive energy estimates for the couple of solutions $(\bsym{u}_{n},\bsym{\xi}^n)$ for $n \in \N.$ We prove below that for the fluid velocity, there exists a positive constant $C_1$ such that:  
\begin{equation}
\label{eq:estimee_energie_u}
\int_{0}^{T}  \norme{\bsym{u}_n}{V}^2dt \leq C_1, \hspace{10pt} \sup_{0 \leq t \leq T}\normeH{\bsym{u}_n}^2 \leq C_1, \quad a.s.
\end{equation}
Concerning the noise basis equation, we prove that for all $\alpha \in (0,1),$ there exists a positive constant $C_2(\alpha)$ such that:
\begin{equation}
\label{eq:estimee_energie_xi}
\int_0^T \sum_{i=0}^{\infty} \norme{\bsym{\xi}^n_i}{1+\alpha}^{2} dt \leq C_2(\alpha), \hspace{10pt} \sum_{i=0}^{\infty} \sup_{0 \leq t \leq T} \norme{\bsym{\xi}_i^n}{\alpha}^2 \leq C_2(\alpha), \quad a.s.
\end{equation}
\begin{proof}
{Applying It\^o's formula to $\normeH{\bsym{u}_n}^2$, using $b(\bsym{u}_n,\bsym{u}_n,\bsym{u}_n)=0$, and integrating by parts in the $F$ term, the It\^o correction cancels the corresponding drift contribution. Moreover the stochastic integral vanishes by incompressibility. Hence}
\begin{equation}
\label{eq:estimee_energie_intermediaire_u}
\normeH{\bsym{u}_n(t)}^2 + \int_0^t \norme{\bsym{u}_n}{V}^{2} ds \leq \normeH{\bsym{u}_n(0)}^2,
\end{equation}
{which gives \eqref{eq:estimee_energie_u}. For the noise modes, testing the second Galerkin equation against $A^\alpha\bsym{\xi}_i^n$ gives}
\begin{equation}
\label{eq:estimee_energie_intermediaire_xi}
\frac12\partial_t\norme{\bsym{\xi}_i^n}{\alpha}^2+\norme{\bsym{\xi}_i^n}{1+\alpha}^2
=-b(\bsym{u}_n,\bsym{\xi}_i^n,A^\alpha\bsym{\xi}_i^n)-b(\bsym{\xi}_i^n,\bsym{u}_n,A^\alpha\bsym{\xi}_i^n).
\end{equation}
{The two-dimensional Sobolev embeddings and interpolation yield}
\begin{equation}
\label{eq:inequality_xi_intermediaire_2d}
\frac12\partial_t\norme{\bsym{\xi}_i^n}{\alpha}^{2}+\norme{\bsym{\xi}_i^n}{1+\alpha}^{2}
\leq \norme{\bsym{\xi}_i^n}{\alpha}\norme{\bsym{u}_n}{V}\norme{\bsym{\xi}_i^n}{1+\alpha}
+\norme{\bsym{u}_n}{\alpha}\norme{\bsym{\xi}_i^n}{\alpha}^{\alpha}\norme{\bsym{\xi}_i^n}{1+\alpha}^{2-\alpha}.
\end{equation}
{Young's inequality and $\norme{\bsym{u}_n}{\alpha}^{2/\alpha}\leq C\normeH{\bsym{u}_n}^{2(1-\alpha)/\alpha}\norme{\bsym{u}_n}{V}^2$ imply
\[
\partial_t\norme{\bsym{\xi}_i^n}{\alpha}^2
\leq C\bigl(1+\normeH{\bsym{u}_n}^{2(1-\alpha)/\alpha}\bigr)\norme{\bsym{u}_n}{V}^2\norme{\bsym{\xi}_i^n}{\alpha}^2.
\]
By \eqref{eq:estimee_energie_intermediaire_u}, we have:
$$
\int_0^T \bigl(1+\normeH{\bsym{u}_n}^{2(1-\alpha)/\alpha}\bigr)\norme{\bsym{u}_n}{V}^2 ds \le c(\alpha, |u_0|_H)
$$
then Gronwall's lemma and summation over $i$ give the $L^\infty(0,T;H^\alpha)$ bound in \eqref{eq:estimee_energie_xi}. 
More precisely,  we have
\begin{equation}
	\label{truc}
\sup_{0 \leq t \leq T} \norme{\bsym{\xi}_i^n}{\alpha}^2  \le c(\alpha, |u_0|_H,T) \norme{\bsym{\xi}_{i,0}^n}{\alpha}^2
\end{equation}
Integrating \eqref{eq:inequality_xi_intermediaire_2d} and using Young's inequality once more gives the $L^2(0,T;H^{1+\alpha})$ bound.}
\end{proof}


Let us now use the energy estimates to prove the tightness of both $\bigl(\mathcal{L}(\bsym{u}_{n})\bigr)_{n \in \N}$ and every sequence $\bigl(\mathcal{L}(\bsym{\xi}_{i}^n)\bigr)_{n \in \N}$ for all $i \in \N$ in specific spaces. For all $n \in \N,$ 
\begin{equation}
\left\{
	\begin{aligned}
&\bsym{u}_{n}(t) = \bsym{u}_{n}(0) - \displaystyle\int_{0}^t A\bsym{u}_n + B^n(\bsym{u}_n,\bsym{u}_n) + \F[n]{\bsym{\xi}^n}{\bsym{u}_n} ds + \displaystyle\int_0^t \G[n]{\bsym{\xi}^n}{\bsym{u}_n}d\bsym{W}_s, \\[0.5em]
&\bsym{\xi}^n_i(t) = \bsym{\xi}^n_i(0) - \displaystyle \int_{0}^t A \bsym{\xi}^n_i + B^n(\bsym{\xi}^n_i,\bsym{u}_n) + B^n(\bsym{u}_n,\bsym{\xi}^n_i)ds, \qquad i \in \N.
	\end{aligned}
	\right.
\end{equation}
\begin{lemme}
\label{lemme:tightness_2d}
The following statements hold:
\begin{enumerate}[label=(\roman*)]
    \item The sequence $\bigl(\mathcal{L}(\bsym{u}_{n})\bigr)_{n \in \N}$ is tight in $L^{2}(0,T,H)$ and $C^0\bigl(0,T,\domA{-\beta}\bigr)$.

    \item The sequence $\bigl(\mathcal{L}(\bsym{\xi}_i^n)_{n \in \N}\bigr)_{i \in \N}$ is tight in $\bigl(L^{2}(0,T,H)\bigr)^\N$ and $\bigl(C^0\big(0,T,\domA{\frac{q}{2}}\big)\bigr)^\N$, $q \in (0,\alpha).$

%
\end{enumerate}
\end{lemme}
\begin{proof}
\label{proof:tightness_2d}
{We only record the estimates needed for compactness. By \eqref{eq:estimee_energie_u}, $(\bsym{u}_n)_n$ is bounded in probability in $L^2(0,T,V)$. Writing the first Galerkin equation as $\bsym{u}_n=J_1^n+\cdots+J_5^n$, estimates \eqref{eq:A-B_majoration}, \eqref{eq:F_majoration}, $H^{1/2}\hookrightarrow L^4$, and \eqref{eq:estimee_energie_xi} give, for every $\gamma>1$,}
\[
\sup_n\norme{J_1^n+J_2^n+J_3^n+J_4^n}{W^{1,2}(0,T,\domA{-\gamma})}<\infty\quad a.s.
\]
{For every $\delta<1/2$, the standard estimate for stochastic convolutions \cite[Lemma~2.1]{FlandoliGatarek1995} yields}
\[
\sup_n\mathbb E\norme{J_5^n}{W^{\delta,2}(0,T,\domA{-\gamma})}^2<\infty.
\]
{The compact embedding $L^2(0,T,V)\cap W^{\delta,2}(0,T,\domA{-\gamma})\Subset L^2(0,T,H)$ proves the first part of (i). For each fixed $i$, the second equation and the same operator estimates give}
\begin{equation}
\label{eq:tightness_xi}
\sup_n\norme{\bsym{\xi}_i^n}{W^{1,2}(0,T,\domA{-\gamma})}<\infty\quad a.s.,
\end{equation}
{while \eqref{eq:estimee_energie_xi} gives boundedness in $L^2(0,T,\domA{(1+\alpha)/2})$. The same compactness criterion proves tightness for each mode in $L^2(0,T,H)$; Tychonoff's theorem then gives the first part of  (ii). Finally, the compact embedding
\[
W^{a,p}(0,T,\domA{-\gamma})\Subset C^0(0,T,\domA{-\beta}),\qquad ap>1,\ \beta>\gamma,
\]
together with the preceding deterministic bounds and the corresponding $W^{a,p}$ estimate for the stochastic convolution, concludes (i). Combining \eqref{eq:tightness_xi} with \eqref{eq:estimee_energie_xi} and interpolation gives (ii), again followed by Tychonoff's theorem.}
\end{proof}


\bigskip

We now  prove the existence of a martingale solution by passing to the limit in the Galerkin approximations. 

Consider $\beta >  1$ The Wiener process verifies $\bsym{W} \in C^0\bigl(0,T,\domA{-\beta}\bigr).$ Define $\bsym{W}^{n} = \bsym{W}.$ We recall that 
$ \bsym{\xi} = (\bsym{\xi}_i)_{i \in \N}$ and $\bsym{\xi}^n = (\bsym{\xi}_i^n)_{i \in \N}.$ \rev{For all $q \in (0,\alpha)$, the family of laws of the triples $(\bsym{u}_n, \bsym{\xi}^n, \bsym{W}^n)$ is tight in} 
\begin{equation*}
\bigl(L^2(0,T,H) \cap C^0\bigl(0,T,\domA{-\beta}\bigr)\bigr) \times \bigl(L^2(0,T,H) \cap C^0\big(0,T,\domA{\frac{q}{2}}\big)\bigr)^\N \times C^0\bigl(0,T,\domA{-\beta}\bigr).
\end{equation*}
\rev{By the Skorokhod representation theorem, there exist a sequence of stochastic bases} $\bigl(\bar{\Omega}, \bar{\mathcal{F}}, \bigl\{\bar{\mathcal{F}}_t\bigr\}_{t \geq 0}, \bar{\mathbb{P}}, \bar{\bsym{W}}^{n}\bigr),$ a sequence of random variables $\bigl(\bar{\bsym{u}}_n, \bar{\bsym{\xi}}^n, \bar{\bsym{W}}^n\bigr)$ taking values in the above space  and a limit triple $\bigl(\bar{\bsym{u}}, \bar{\bsym{\xi}}^{\infty}, \bar{\bsym{W}}\bigr)$ such that $\bigl(\bar{\bsym{u}}_n, \bar{\bsym{\xi}}^n, \bar{\bsym{W}}^n \bigr)$ has the same law as $\bigl(\bsym{u}_n, \bsym{\xi}^n, \bsym{W}^n\bigr).$ All these sequences converge up to a subsequence
\begin{equation}
\label{eq:convergence}
\begin{aligned}
\bar{\bsym{u}}_n &\longrightarrow \bar{\bsym{u}} \quad \text{in} \quad L^2(0,T,H) \times C^0\bigl(0,T,\domA{-\beta}\bigr)  \quad \bar{\mathbb{P}}-a.s. \\
\bar{\bsym{\xi}}^n &\longrightarrow \bar{\bsym{\xi}}^{\infty} \quad \text{in} \quad \bigl(L^2(0,T,H) \times C^0\big(0,T,\domA{\frac{q}{2}}\big)\bigr)^\N \quad \bar{\mathbb{P}}-a.s. \\
\bar{\bsym{W}}^n &\longrightarrow \bar{\bsym{W}} \quad \text{in} \quad C^0\bigl(0,T,\domA{-\beta}\bigr) \quad \bar{\mathbb{P}}-a.s.
\end{aligned}
\end{equation} 
Arguing as in \cite{Bensoussan1995}, especially to treat the stochastic integral, we know that $(\bar{\bsym{u}}_n, \bar{\bsym{\xi}}^n)$ verifies the following equation, for all $n \in \N,$
\begin{equation}
\label{eq:equation_bar}
\left\{
	\begin{aligned}
&\bar{\bsym{u}}_{n}(t) = \bar{\bsym{u}}_{n}(0) - \displaystyle\int_{0}^t A\bar{\bsym{u}}_n + B^n(\bar{\bsym{u}}_n,\bar{\bsym{u}}_n) + \F[n]{\bar{\bsym{\xi}}^n}{\bar{\bsym{u}}_n} ds + \displaystyle\int_0^t \G[n]{\bar{\bsym{\xi}}^n}{\bar{\bsym{u}}_n}d\bar{\bsym{W}}^n_s, \\[0.5em]
&\bar{\bsym{\xi}}^n_i(t) = \bar{\bsym{\xi}}^n_i(0) - \displaystyle \int_{0}^t A \bar{\bsym{\xi}}^n_i + B^n(\bar{\bsym{\xi}}^n_i,\bar{\bsym{u}}_n) + B^n(\bar{\bsym{u}}_n,\bar{\bsym{\xi}}^n_i) ds, \qquad i \in \N.
	\end{aligned}
	\right.
\end{equation}
Also the couple $(\bar{\bsym{u}}_n, \bar{\bsym{\xi}}^n)$ verifies the same energy estimates as $(\bsym{u}_n, \bsym{\xi}^n)$. For $\bar{\bsym{u}}_n,$ we have for all $p >1:$  
\begin{equation}
\label{eq:estimation_u_bar}
\int_{0}^{T}  \norme{\bar{\bsym{u}}_n}{V}^2dt \leq C_1, \hspace{10pt} \sup_{0 \leq t \leq T}\normeH{\bar{\bsym{u}}_n}^2 \leq C_1, \quad a.s.
\end{equation}
Fix $i \in \N.$ We have the estimate on $\bar{\bsym{\xi}}_i^n$ for all $\alpha \in (0,1):$
\begin{equation}
\label{eq:estimation_xi_bar}
\int_0^T \sum_{i=0}^{\infty} \norme{\bar{\bsym{\xi}}^n_i}{1+\alpha}^{2} dt \leq C_2(\alpha), \hspace{10pt} \sum_{i=0}^{\infty} \sup_{0 \leq t \leq T} \norme{\bar{\bsym{\xi}}_i^n}{\alpha}^2 \leq C_2(\alpha), \quad a.s.
\end{equation}
By weak lower semicontinuity of the norm, together with Fatou's lemma with respect to the mode index $i$, the uniform \(L^\infty(0,T;D(A^{\alpha/2}))\) and \(L^2(0,T;D(A^{(1+\alpha)/2}))\) bounds on the Galerkin solutions pass to the limit as \(n\to\infty\).
Therefore, $\forall \alpha \in (0,1),$  $p > 1,$
\begin{equation}
	\label{eq:estimation_limite}
	\begin{aligned}
		\int_{0}^{T}  \norme{\bar{\bsym{u}}}{V}^2dt \leq C_1, &\hspace{10pt} \sup_{0 \leq t \leq T}\normeH{\bar{\bsym{u}}}^p \leq C_1, \quad a.s. \\
		\int_0^T \sum_{i=0}^{\infty} \norme{\bar{\bsym{\xi}}_i^{\infty}}{1+\alpha}^{2} dt \leq C_2(\alpha), &\hspace{10pt} \sum_{i=0}^{\infty} \sup_{0 \leq t \leq T} \norme{\bar{\bsym{\xi}}_i^{\infty}}{\alpha}^2 \leq C_2(\alpha), \quad a.s.
	\end{aligned}
\end{equation}
and, thanks to \eqref{eq:convergence},  $\bar{\mathbb{P}}-a.s.$
\begin{align*}
	&\bar{\bsym{u}}(.,\omega) \in L^2\bigl(0,T,V\bigr) \cap L^{\infty}\bigl(0,T,H\bigr) \cap C^0\bigl(0,T,\domA{-\gamma}\bigr),  \\
	&\bigl(\bar{\bsym{\xi}}_i^{\infty}(.,\omega)\bigr)_{i \in \N} \in l^2\Bigl(\N, L^2\big(0,T,\domA{\frac{1+\alpha}{2}}\big) \cap L^{\infty}\big(0,T,\domA{\frac{\alpha}{2}}\big) \cap C^0\bigl(0,T,\domA{\frac{q}{2}}\bigr)\Bigl).
\end{align*}
We have a candidate solution $\big(\bar{\bsym{u}}, \bar{\bsym{\xi}}^{\infty}\big)$. We prove here that it is a solution, that is we pass to the limit in \eqref{eq:equation_bar}. Consider the variational problem associated. For all $n \in \N, 0\leq t \leq T, \bsym{y},\bsym{z} \in \domA{\gamma}:$
\begin{equation}
\label{eq:pb_variationel}
\left\{
	\begin{aligned}
\scalar{\bar{\bsym{u}}_{n}(t) - \bar{\bsym{u}}_{n}(0)}{\bsym{y}}{H} + \displaystyle\int_{0}^t \scalar{A\bar{\bsym{u}}_n}{\bsym{y}}{H} + \scalar{B^n(\bar{\bsym{u}}_n,\bar{\bsym{u}}_n)}{\bsym{y}}{H} \\
+ \scalar{\F[n]{\bar{\bsym{\xi}}^n}{\bar{\bsym{u}}_n}}{\bsym{y}}{H} ds = \scalar[\bigg]{\displaystyle\int_0^t \G[n]{\bar{\bsym{\xi}}^n}{\bar{\bsym{u}}_n}d\bar{\bsym{W}}^n_s}{\bsym{y}}{H}, \\[0.5em] 
\scalar{\bar{\bsym{\xi}}^n_i(t) - \bar{\bsym{\xi}}^n_i(0)}{\bsym{z}}{H} + \displaystyle \int_{0}^t \scalar{A \bar{\bsym{\xi}}^n_i}{\bsym{z}}{H} + \scalar{B^n(\bar{\bsym{\xi}}^n_i,\bar{\bsym{u}}_n)}{\bsym{z}}{H} \\
+ \scalar{B^n(\bar{\bsym{u}}_n,\bar{\bsym{\xi}}^n_i)}{\bsym{z}}{H} ds = 0, \qquad i \in \N.
	\end{aligned}
	\right.
\end{equation}
We fix $t \in [0,T]$ and pass to the limit in \eqref{eq:pb_variationel}. Let us rewrite the system for the sake of simplicity:
\begin{equation}
\left\{
\begin{aligned}
&J_{1}^n + J_{2}^n + J_{3}^n + J_{4}^n = J_{5}^n, \\[0.5em]
&K_{1}^n + K_{2}^n + K_{3}^n + K_{4}^n = 0.
\end{aligned}
\right.
\end{equation}

\begin{lemme}
\label{lemme:convergence}
Consider the stochastic basis $\bigl(\bar{\Omega}, \bar{\mathcal{F}}, \bigl\{\bar{\mathcal{F}}_t\bigr\}_{t \geq 0}, \bar{\mathbb{P}}, \bar{\bsym{W}}^{n}\bigr).$ \rev{Up to a subsequence, the following almost sure convergences hold:} 

\begin{enumerate}[label=(\roman*)]
    \item $J_i^n \longrightarrow J_i \quad \bar{\mathbb{P}}-a.s. \quad i \in \{1,..,5\}.$
    \item $K_i^n \longrightarrow K_i \quad \bar{\mathbb{P}}-a.s. \quad i \in \{1,..,4\}.$
\end{enumerate}
\end{lemme}

\begin{proof}
\begin{enumerate}[label=(\roman*)]

\item For $J_{1}^n$ we write 
\begin{equation*}
\scalar{\bar{\bsym{u}}_{n}(t) - \bar{\bsym{u}}_{n}(0)}{\bsym{y}}{H} = \big(\normeH{\bar{\bsym{u}}_{n}(t) - \bar{\bsym{u}}(t)} +\normeH{\bar{\bsym{u}}_{n}(0) - \bar{\bsym{u}}(0)} \big) \normeH{\bsym{y}}.
\end{equation*}
Thanks to \eqref{eq:convergence} $J_{1}^n \underset{n \rightarrow \infty}{\longrightarrow} J_{1}.$ Now $J_{2}^n.$  Since $A$ is symmetric,
\begin{equation*}
\int_0^t \scalar{A\bar{\bsym{u}}_n -A\bar{\bsym{u}}}{\bsym{y}}{H} ds \leq \sqrt{t}\norme{\bar{\bsym{u}}_n -\bar{\bsym{u}}}{L^2(0,T,H)}\norme{A\bsym{y}}{H}.
\end{equation*}
And again thanks to \eqref{eq:convergence}, $J_{2}^n \underset{n \rightarrow \infty}{\longrightarrow} J_{2}.$ Now $J_{3}^n.$ We write:
\begin{align*}
\int_0^t \scalar{B^n(\bar{\bsym{u}}_n,\bar{\bsym{u}}_n)- B(\bar{\bsym{u}},\bar{\bsym{u}})}{\bsym{y}}{H} ds &= \int_0^t \scalar{B(\bar{\bsym{u}}_n, \bar{\bsym{u}}_n)}{P_n\bsym{y}- \bsym{y}}{H} ds \\
&+ \int_0^t \scalar{B(\bar{\bsym{u}}_n, \bar{\bsym{u}}_n) - B(\bar{\bsym{u}}, \bar{\bsym{u}})}{\bsym{y}}{H} ds \\
&\leq J_{3,1}^n +J_{3,2}^n.
\end{align*}
Since $P_n\bsym{y}- \bsym{y} \in \domA{\gamma},$ we have:
\begin{align*}
J_{3,1}^n &\leq \int_0^t \norme{B(\bar{\bsym{u}}_n,\bar{\bsym{u}}_n)}{\domA{-\gamma}} ds \norme{P_n\bsym{y}-\bsym{y}}{\domA{\gamma}}\\
&\leq C \norme{\bar{\bsym{u}}_n}{L^{2}(0,T,H)}^2 \norme{P_n\bsym{y}-\bsym{y}}{\domA{\gamma}}.
\end{align*}
Thanks to \eqref{eq:estimation_u_bar} and \eqref{eq:convergence_norme_proj}, $J_{3,1}^n \underset{n \rightarrow \infty}{\longrightarrow} 0.$ Now we deal with $J_{3,2}^n.$ Because $B$ is bilinear, we can write: 
\begin{align*}
J_{3,2}^n &= \int_0^t \scalar{B(\bar{\bsym{u}}_n- \bar{\bsym{u}}, \bar{\bsym{u}}_n) + B(\bar{\bsym{u}}, \bar{\bsym{u}}_n - \bar{\bsym{u}})}{\bsym{y}}{H} ds \\
&\leq C\norme{\bsym{y}}{\domA{\gamma}} \int_0^t \normeH{\bar{\bsym{u}}_n - \bar{\bsym{u}}} \normeH{\bar{\bsym{u}}_n}  + \normeH{\bar{\bsym{u}}}\normeH{\bar{\bsym{u}}_n-\bar{\bsym{u}}} ds \\
&\leq C\norme{\bsym{y}}{\domA{\gamma}} \norme{\bar{\bsym{u}}_n-\bar{\bsym{u}}}{L^2(0,T,H)} \Bigl[\sup_n \norme{\bar{\bsym{u}}_n}{L^2(0,T,H)} +  \norme{\bar{\bsym{u}}}{L^2(0,T,H)} \Bigr].
\end{align*}
Thanks to \eqref{eq:convergence}, $\norme{\bar{\bsym{u}}_n-\bar{\bsym{u}}}{L^2(0,T,H)} \underset{n \rightarrow \infty}{\longrightarrow} 0.$ Since $(\bar{\bsym{u}}_n)_{n \in \N}$ converges, it is bounded. That proves $J_{3,2}^n \underset{n \rightarrow \infty}{\longrightarrow} 0,$ hence the convergence of $J_{3}^n$ towards $J_{3}.$ For the term $J_{4}^n,$ we compute:
\begin{equation*}
\begin{aligned}
\int_0^t \scalar{\F[n]{\bar{\bsym{\xi}}^n}{\bar{\bsym{u}}_n}}{\bsym{y}}{H} - \scalar{\F{\bar{\bsym{\xi}}^{\infty}}{\bar{\bsym{u}}}}{\bsym{y}}{H} ds = \int_0^t \scalar{\F{\bar{\bsym{\xi}}^n}{\bar{\bsym{u}}_n}}{P_n\bsym{y}-\bsym{y}}{H} \\
+ \scalar{\F{\bar{\bsym{\xi}}^n}{\bar{\bsym{u}}_n} - \F{\bar{\bsym{\xi}}^{\infty}}{\bar{\bsym{u}}_n}}{\bsym{y}}{H} \\
+ \scalar{\F{\bar{\bsym{\xi}}^{\infty}}{\bar{\bsym{u}}_n}- \F{\bar{\bsym{\xi}}^{\infty}}{\bar{\bsym{u}}}}{\bsym{y}}{H} ds \\
= J_{4,1}^n + J_{4,2}^n +J_{4,3}^n.
\end{aligned}
\end{equation*}
Start with $J_{4,1}^n.$ Thanks to \eqref{eq:F_majoration} and \eqref{eq:estimation_xi_bar},
\begin{align*}
J_{4,1}^n &\leq \norme{P_n\bsym{y}-\bsym{y}}{\domA{\gamma}} \int_0^t \norme{\F{\bar{\bsym{\xi}}^{n}}{\bar{\bsym{u}}_n}}{\domA{-\gamma}} ds \\
&\leq C\norme{P_n\bsym{y}-\bsym{y}}{\domA{\gamma}} \int_0^t \norme{\bar{\bsym{u}}_n}{V}^2 ds. 
\end{align*}
Thanks to \eqref{eq:estimation_u_bar} and \eqref{eq:convergence_norme_proj}, up to a subsequence and $J_{4,1}^n \underset{n \rightarrow \infty}{\longrightarrow} 0 \quad \bar{\mathbb{P}}-a.s.$ For $J_{4,2}^n,$ thanks to \eqref{eq:F_majoration}, we have 
\begin{equation*}
\begin{aligned}
\int_0^t \scalar{\F{\bar{\bsym{\xi}}^n}{\bar{\bsym{u}}_n} - \F{\bar{\bsym{\xi}}^{\infty}}{\bar{\bsym{u}}_n}}{\bsym{y}}{H} ds \leq \norme{\nabla \bsym{ y}}{L^{\infty}}\int_0^t \norme{\bar{\bsym{u}}_n}{V} \\ 
\normeH{\sum_{i=0}^{\infty}\big(\bar{\bsym{\xi}}_i^n\bar{\bsym{\xi}}_i^{n\intercal} - \bar{\bsym{\xi}}_i^{\infty}\bar{\bsym{\xi}}_i^{\infty \intercal}\big)} ds. 
\end{aligned}
\end{equation*}
We are left with the term under the integral. \rev{Applying Hölder's inequality,} 
\begin{align*}
\sum_{i=0}^{\infty} \normeH{\bigl(\bar{\bsym{\xi}}_i^n\bar{\bsym{\xi}}_i^{n\intercal} - \bar{\bsym{\xi}}_i^{\infty}\bar{\bsym{\xi}}_i^{\infty \intercal}\bigr)} ds &=  \sum_{i=0}^{\infty} \normeH{\bar{\bsym{\xi}}_i^n \bigl(\bar{\bsym{\xi}}_i^{n\intercal}-\bar{\bsym{\xi}}_i^{\infty \intercal}\bigr) + \bigl(\bar{\bsym{\xi}}_i^n-\bar{\bsym{\xi}}_i^{\infty}\bigr)\bar{\bsym{\xi}}_i^{\infty \intercal}} \\
&\leq 2 \sum_{i=0}^{\infty} \norme{\bar{\bsym{\xi}}_i^n}{L^4} \norme{\bar{\bsym{\xi}}_i^n - \bar{\bsym{\xi}}_i^{\infty}}{L^4}.
\end{align*}
We take $\alpha = \frac{1}{2}.$ By Sobolev embedding, $H^\alpha \hookrightarrow L^{4}.$ Therefore, we have the following inequality: 
\begin{equation*}
\sum_{i=0}^{\infty} \norme{\bar{\bsym{\xi}}_i^n}{L^4} \norme{\bar{\bsym{\xi}}_i^n - \bar{\bsym{\xi}}_i^{\infty}}{L^4} \leq \sum_{i=0}^{\infty} \norme{\bar{\bsym{\xi}}_i^n}{1/2} \norme{\bar{\bsym{\xi}}_i^n - \bar{\bsym{\xi}}_i^{\infty}}{1/2}.
\end{equation*}
By Cauchy-Schwarz inequality: 
\begin{equation*}
\sum_{i=0}^{\infty} \norme{\bar{\bsym{\xi}}_i^n}{1/2} \norme{\bar{\bsym{\xi}}_i^n - \bar{\bsym{\xi}}_i^{\infty}}{1/2} \leq \Big(\sum_{i=0}^{\infty} \norme{\bar{\bsym{\xi}}_i^n}{1/2}^2\Big)^{\frac{1}{2}} \Big(\sum_{i=0}^{\infty} \norme{\bar{\bsym{\xi}}_i^n - \bar{\bsym{\xi}}_i^{\infty}}{1/2}^2\Big)^{\frac{1}{2}}.
\end{equation*}
Thanks to \eqref{eq:estimation_xi_bar}, the first sum is bounded. By dominated convergence, \eqref{truc} and the assumption the initial data, the second term goes to zero. We can conclude that $J_{4,2}^n \underset{n \rightarrow \infty}{\longrightarrow} 0.$ The last term is $J_{4,3}^n:$
\begin{align*}
\int_0^t \scalar{\F{\bar{\bsym{\xi}}^{\infty}}{\bar{\bsym{u}}_n} - \F{\bar{\bsym{\xi}}^{\infty}}{\bar{\bsym{u}}}}{\bsym{y}}{H} ds &= \int_0^t \scalar{\bar{\bsym{u}}_n- \bar{\bsym{u}}}{\F{\bar{\bsym{\xi}}^{\infty}}{\bar{\bsym{y}}}}{H} ds \\
&\leq C\int_0^t \normeH{\F{\bar{\bsym{\xi}}^{\infty}}{\bar{\bsym{y}}}}^2 ds \norme{\bar{\bsym{u}}_n - \bar{\bsym{u}}}{L^2(0,T,H)}.
\end{align*}
And this goes to $0$ thanks to \eqref{eq:convergence}. We proved that $J_{4}^n \underset{n \rightarrow \infty}{\longrightarrow} J_{4}.$ For the convergence of $J_{5}^n,$ we apply the following Lemma (the reader can refer to  \cite{Debussche_Glatt-Holtz_Temam} for further details):
\begin{lemme}
\label{lemme:convergence_stochastique}
	  	Consider a sequence of stochastic basis $S_n = \bigl(\Omega, \mathcal{F}, \{\mathcal{F}_t^n\}_{t \in [0,T]}, \mathbb{P}, \bsym{W}^n\bigr)$ with $\bsym{W}^n$ a cylindrical Wiener process (over $U$) with respect to $\mathcal{F}_t^n.$ Assume that $(\psi^n)_{n \geq 1}$ is a collection of $X-$ valued $\mathcal{F}_t^n-$predictable processes such that $\bsym{\psi}^n \in L^2\bigl(0,T,\mathcal{L}_2(U,X)\bigr)$ a.s. Finally consider $S = \bigl(\Omega, \mathcal{F}, \mathbb{P}, \bsym{W}\bigr)$ with $\bsym{W}$ a cylindrical Wiener process (over $U$) and $\bsym{\psi} \in L^2\bigl(0,T,\mathcal{L}_2(U,X)\bigr) $ which is $\mathcal{F}_t-$predictable. If: 
	  	\begin{center}
	  		$\bsym{W}^n \underset{n \rightarrow \infty}{\overset{\bar{\mathbb{P}}}{\longrightarrow}} \bsym{W} \quad in \quad C^0(0,T,U_0),$ \\
	  		$\bsym{\psi}^n\underset{n \rightarrow \infty}{\overset{\bar{\mathbb{P}}}{\longrightarrow}} \bsym{\psi} \quad in \quad L^2\bigl(0,T,\mathcal{L}_2(U,X)\bigr),$ 
	  	\end{center}
Then, 
	    \begin{center}
	  	    $\displaystyle\int_0^t \bsym{\psi}^n d\bsym{W}_s^n \underset{n \rightarrow \infty}{\overset{\bar{\mathbb{P}}}{\longrightarrow}} \displaystyle\int_0^t \bsym{\psi} d\bsym{W}_s \quad in \quad L^2(0,T,X).$ 
	    \end{center}
	  
\end{lemme}
Consider $\tilde{\gamma} > \gamma.$ We take $\bsym{\psi}^n = \G{\bar{\bsym{\xi}}^n}{\bar{\bsym{u}}_n} \in L^2\bigl(0,T,\mathcal{L}_2\bigl(L^2,\domA{-\tilde{\gamma}}\bigr)\bigr).$ What we have to prove is the convergence $\G[n]{\bar{\bsym{\xi}}^n}{\bar{\bsym{u}}_n} \overset{\bar{\mathbb{P}}}{\underset{n \rightarrow \infty}{\longrightarrow}} \G{\bar{\bsym{\xi}}^{\infty}}{\bar{\bsym{u}}} $ in $L^2\bigl(0,T,\mathcal{L}_2\bigl(L^2,\domA{-\tilde{\gamma}}\bigr)\bigr):$
\begin{equation*}
    \begin{aligned}
\int_0^t \norme{\G[n]{\bar{\bsym{\xi}}^n}{\bar{\bsym{u}}_n} - \G{\bar{\bsym{\xi}}^{\infty}}{\bar{\bsym{u}}}}{\mathcal{L}_2\big(L^2,\domA{-\tilde{\gamma}}\big)}^2 ds  \leq \\
3 \int_0^t \norme{\G[n]{\bar{\bsym{\xi}}^n}{\bar{\bsym{u}}_n} - \G{\bar{\bsym{\xi}}^n}{\bar{\bsym{u}}_n}}{\mathcal{L}_2\bigl(L^2,\domA{-\tilde{\gamma}}\bigr)}^2 \\
+  \norme{\G{\bar{\bsym{\xi}}^n}{\bar{\bsym{u}}_n} - \G{\bar{\bsym{\xi}}^n}{\bar{\bsym{u}}}}{\mathcal{L}_2\bigl(L^2,\domA{-\tilde{\gamma}}\bigr)}^2  \\
+ \norme{\G{\bar{\bsym{\xi}}^n}{\bar{\bsym{u}}} - \G{\bar{\bsym{\xi}}^{\infty}}{\bar{\bsym{u}}}}{\mathcal{L}_2\bigl(L^2,\domA{-\tilde{\gamma}}\bigr)}^2 ds \\
= a_n + b_n + c_n. 
\end{aligned}
\end{equation*}
Start with $a_n.$ Thanks to \eqref{eq:majoration_projecteur}, we have:
\begin{align*}
a_n &= \int_0^t \norme{\big(P_n-I\big)\G{\bar{\bsym{\xi}}^n}{\bar{\bsym{u}}_n}}{\mathcal{L}_2\bigl(L^2,\domA{-\tilde{\gamma}}\bigr)}^2 ds \\
&\leq \int_0^t \lambda_{n+1}^{-(\gamma - \tilde{\gamma})} \norme{\G{\bar{\bsym{\xi}}^n}{\bar{\bsym{u}}_n}}{\mathcal{L}_2\bigl(L^2,\domA{-\gamma}\bigr)}^2ds .
\end{align*}
Hence thanks to \eqref{eq:A-B_majoration} and \eqref{eq:G_majoration},
\begin{align*}
a_n &\leq \frac{1}{\lambda_{n+1}^{\gamma - \tilde{\gamma}}} \int_0^t \sum_{i=0}^{\infty} \norme{B(\bar{\bsym{\xi}}_i^n, \bar{\bsym{u}}_n)}{\domA{-\gamma}}^2 ds. \\
&\leq \frac{1}{\lambda_{n+1}^{\gamma - \tilde{\gamma}}} \int_0^t \sum_{i=0}^{\infty} \normeH{\bar{\bsym{\xi}}_i^n}^2 \normeH{\bar{\bsym{u}}_n}^2 ds.
\end{align*}
And we conclude thanks to \eqref{eq:estimation_limite} and the fact that $(\lambda_i)_{i \in \N}$ is an increasing unbounded sequence that $a_n \underset{n \rightarrow \infty}{\longrightarrow} 0.$ Now $b_n.$ Thanks to \eqref{eq:estimation_xi_bar},
\begin{align*}
b_n &\leq \int_0^t \sum_{i=0}^{\infty} \norme{B(\bar{\bsym{\xi}}_i^n, \bar{\bsym{u}}_n - \bar{\bsym{u}})}{\domA{-\tilde{\gamma}}}^2 ds \\
&\leq C \norme{\bar{\bsym{u}}_n - \bar{\bsym{u}}}{L^2(0,T,H)}^2.
\end{align*}
And $b_n \underset{n \rightarrow \infty}{\longrightarrow} 0.$ As for $c_n:$
\begin{align*}
c_n &\leq \int_0^t \sum_{i=0}^{\infty} \norme{B(\bar{\bsym{\xi}}_i^n-\bar{\bsym{\xi}}_i^{\infty}, \bar{\bsym{u}})}{\domA{-\tilde{\gamma}}}^2 ds \\
&\leq \normeH{\bar{\bsym{u}}}^2  \sum_{i=0}^{\infty} \norme{\bar{\bsym{\xi}}_i^n-\bar{\bsym{\xi}}_i^{\infty}}{L^2(0,T,H)}^2.
\end{align*}
And by dominated convergence, thanks to \eqref{eq:estimation_xi_bar}, $c_n \underset{n \rightarrow \infty}{\longrightarrow} 0.$ Therefore, $J_{5}^n \underset{n \rightarrow \infty}{\longrightarrow} J_{5}$ in probability in $L^2\bigl(0,T,\domA{-\tilde{\gamma}}\bigr).$ Note that convergence occurs in $L^2\bigl(0,T,\domA{-\tilde{\gamma}}\bigr)$ rather than in $L^2\bigl(0,T,\domA{-\gamma}\bigr).$ Since $\gamma > 1$ is arbitrary, this does not affect the tightness or the convergence results.
\item We now study the terms in the second equation. We start with $K_{1}^n,$ and fix $i \in \N:$ 
\begin{equation*}
\scalar{\bar{\bsym{\xi}}^n_i(t) - \bar{\bsym{\xi}}_i^{\infty}(t)}{\bsym{z}}{H} \leq \norme{\bsym{z}}{\domA{\gamma}} \norme{\bar{\bsym{\xi}}^n_i - \bar{\bsym{\xi}}_i^{\infty}}{C^0\big(0,T, \domA{-\gamma}\big)}.
\end{equation*}
And thanks to \eqref{eq:convergence}, the above scalar product goes to $0$ with $n.$ Moreover, 
\begin{align*}
\scalar{P_n\bar{\bsym{\xi}}^n_i(0) - \bar{\bsym{\xi}}_i^{\infty}(0)}{\bsym{z}}{H} &= \scalar{\bar{\bsym{\xi}}^n_i(0)}{P_n\bsym{z}-\bsym{z}}{H} + \scalar{\bar{\bsym{\xi}}^n_i(0) - \bar{\bsym{\xi}}_i^{\infty}(0)}{\bsym{z}}{H} \\
&\leq \normeH{\bar{\bsym{\xi}}^n_i(0)} \normeH{P_n\bsym{z}-\bsym{z}} + \norme{\bar{\bsym{\xi}}^n_i - \bar{\bsym{\xi}}_i^{\infty}}{C^0\bigl(0,T, \domA{-\gamma}\bigr)} \normeH{\bsym{z}}.
\end{align*}
And this goes to $0$ with $n$ thanks to \eqref{eq:convergence_norme_proj} and \eqref{eq:convergence}. Hence $K_{1}^n \underset{n \rightarrow \infty}{\longrightarrow} K_{1}.$ Next $K_{2}^n.$ Since $A$ is symmetric, we can write
\begin{equation*}
\int_0^t \scalar{A\bar{\bsym{\xi}}^n_i - A\bar{\bsym{\xi}}_i^{\infty}}{\bsym{z}}{H} ds \leq \norme{\bar{\bsym{\xi}}^n_i - \bar{\bsym{\xi}}_i^{\infty}}{L^2(0,T,H)} \normeH{A\bsym{z}}.
\end{equation*}
This proves that $K_{2}^n \underset{n \rightarrow \infty}{\longrightarrow} K_{2}$ thanks to \eqref{eq:convergence}. For $K_{3}^n,$ write
\begin{align*}
\int_0^t \scalar{B^n(\bar{\bsym{u}}_n, \bar{\bsym{\xi}}^n_i)}{\bsym{z}}{H} - \scalar{B(\bar{\bsym{u}}, \bar{\bsym{\xi}}_i^{\infty})}{\bsym{z}}{H} ds &= \int_0^t \scalar{B(\bar{\bsym{u}}_n, \bar{\bsym{\xi}}^n_i)}{P_n\bsym{z}-\bsym{z}}{H} \\
&+ \scalar{B(\bar{\bsym{u}}_n, \bar{\bsym{\xi}}^n_i) - B(\bar{\bsym{u}}, \bar{\bsym{\xi}}_i^{\infty})}{\bsym{z}}{H} ds \\
&= K_{3,1}^n + K_{3,2}^n.
\end{align*}
For the first term $K_{3,1}^n,$ we can write: 
\begin{align*}
\int_0^t \scalar{B(\bar{\bsym{u}}_n, \bar{\bsym{\xi}}^n_i)}{P_n\bsym{z}-\bsym{z}}{H} ds &\leq \int_0^t \norme{B(\bar{\bsym{u}}_n, \bar{\bsym{\xi}}^n_i)}{\domA{-\gamma}} \norme{P_n\bsym{z}-\bsym{z}}{\domA{\gamma}} ds \\
&\leq \int_0^t \normeH{\bar{\bsym{u}}_n} \normeH{\bar{\bsym{\xi}}^n_i} \norme{P_n\bsym{z}-\bsym{z}}{\domA{\gamma}} ds \\
&\leq  \norme{P_n\bsym{z}-\bsym{z}}{\domA{\gamma}} \norme{\bar{\bsym{u}}_n}{L^2(0,T,H)} \norme{\bar{\bsym{\xi}}^n_i}{L^2(0,T,H)}.
\end{align*}
Since \eqref{eq:convergence} holds, both $\norme{\bar{\bsym{u}}_n}{L^2(0,T,H)}$  and $\norme{\bar{\bsym{\xi}}^n_i}{L^2(0,T,H)}$ are uniformly bounded in $n.$ Thanks to \eqref{eq:convergence_norme_proj}, $K_{3,1}^n$ goes to $0$ with $n.$ As for the second term $K_{3,2}^n,$ we can write by bilinearity of $B$ that 
\begin{equation*}
B(\bar{\bsym{u}}_n, \bar{\bsym{\xi}}^n_i) - B(\bar{\bsym{u}}, \bar{\bsym{\xi}}_i^{\infty}) = B(\bar{\bsym{u}}_n - \bar{\bsym{u}}, \bar{\bsym{\xi}}^n_i) - B(\bar{\bsym{u}}, \bar{\bsym{\xi}}_i^n - \bar{\bsym{\xi}}_i^{\infty}).
\end{equation*}
Thus we can evaluate that: 
\begin{align*}
K_{3,2}^n &\leq \int_0^t \norme{\bsym{z}}{\domA{}}\Big[\normeH{\bar{\bsym{u}}_n - \bar{\bsym{u}}}\normeH{\bar{\bsym{\xi}}^n_i} + \normeH{\bar{\bsym{u}}}\normeH{\bar{\bsym{\xi}}_i^n - \bar{\bsym{\xi}}_i^{\infty}}  \Big] ds \\
&\leq \norme{\bsym{z}}{\domA{}} \norme{\bar{\bsym{u}}_n - \bar{\bsym{u}}}{L^2(0,T,H)}\norme{\bar{\bsym{\xi}}^n_i}{L^2(0,T,H)} + \norme{\bar{\bsym{u}}}{C^0(0,T,H)}\norme{\bar{\bsym{\xi}}_i^n - \bar{\bsym{\xi}}_i^{\infty}}{L^2(0,T,H)}.
\end{align*}
Thanks to \eqref{eq:convergence}, $K_{3,2}^n$ goes to $0$ with $n.$ The calculation for the last term $K_4^n$ is identical.
\end{enumerate}
\end{proof}
This proves that when we pass to the limit in \eqref{eq:pb_variationel}, we obtain for all $\bsym{y}, \bsym{z}\in \domA{\gamma}$ and almost surely in $(t,\omega) \in [0,T] \times \bar{\Omega}$
\begin{equation}
\label{eq:pb_variationel_solution}
\left\{
	\begin{aligned}
\scalar{\bar{\bsym{u}}(t) - \bar{\bsym{u}}(0)}{\bsym{y}}{H} + \displaystyle\int_{0}^t \scalar{A\bar{\bsym{u}}}{\bsym{y}}{H} + \scalar{B(\bar{\bsym{u}},\bar{\bsym{u}})}{\bsym{y}}{H} \\
+ \scalar{\F{\bar{\bsym{\xi}}^{\infty}}{\bar{\bsym{u}}}}{\bsym{y}}{H} ds = \scalar[\bigg]{\displaystyle\int_0^t \G{\bar{\bsym{\xi}}^{\infty}}{\bar{\bsym{u}}}d\bar{\bsym{W}}_s}{\bsym{y}}{H}, \\[0.5em]
\scalar{\bar{\bsym{\xi}}_i^{\infty}(t) - \bar{\bsym{\xi}}_i^{\infty}(0)}{\bsym{z}}{H} + \displaystyle \int_{0}^t \scalar{A \bar{\bsym{\xi}}_i^{\infty}}{\bsym{z}}{H} + \scalar{B(\bar{\bsym{\xi}}_i^{\infty},\bar{\bsym{u}})}{\bsym{z}}{H} \\
+ \scalar{B(\bar{\bsym{u}},\bar{\bsym{\xi}}_i^{\infty})}{\bsym{z}}{H} ds = 0, \qquad i \in \N.
	\end{aligned}
	\right.
\end{equation}
This is the definition of a martingale solution. Hence the equality, in the space $\domA{-\gamma},$ for almost all $(t,\omega) \in [0,T]\times \bar{\Omega}$  
\begin{equation}
\label{eq:solution}
\left\{
	\begin{aligned}
&\bar{\bsym{u}}(t) - \bar{\bsym{u}}(0) + \displaystyle\int_{0}^t A\bar{\bsym{u}} + B(\bar{\bsym{u}},\bar{\bsym{u}})
+ \F{\bar{\bsym{\xi}}^{\infty}}{\bar{\bsym{u}}} ds = \displaystyle\int_0^t \G{\bar{\bsym{\xi}}^{\infty}}{\bar{\bsym{u}}}d\bar{\bsym{W}}_s, \\[0.5em]
&\bar{\bsym{\xi}}_i^{\infty}(t) - \bar{\bsym{\xi}}^{\infty}_i(0) + \displaystyle \int_{0}^t A \bar{\bsym{\xi}}_i^{\infty} + B(\bar{\bsym{\xi}}_i^{\infty},\bar{\bsym{u}}) + B(\bar{\bsym{u}},\bar{\bsym{\xi}}_i^{\infty}) ds = 0, \qquad i \in \N.
	\end{aligned}
	\right.
\end{equation}
We have thus proved the existence statement in Theorem \ref{thm:well_posedness_2d} (i) and equation \eqref{eq:equation_operateurs} admits at least a martingale solution.

\section{Pathwise solution}
\label{sec:pathwise_solution}

{In this section, we prove pathwise uniqueness and then obtain existence on the prescribed stochastic basis by the Gy\"ongy--Krylov argument. This requires a stronger estimate for the noise modes. Let $s>1$ and choose $\alpha\in(1,s)$ such that $\sum_{i=0}^{\infty}\norme{\bsym{\xi}_{i,0}}{\alpha}^2<\infty$, as assumed in Theorem \ref{thm:well_posedness_2d}(ii). We have the following energy estimates:}
\begin{equation}
\label{eq:estimee_energie_xi_s>1}
\int_0^T \sum_{i=0}^{\infty} \norme{\bsym{\xi}_i^n}{s+\alpha}^{2} dt \leq C_3(\alpha), \hspace{10pt} \sum_{i=0}^{\infty} \sup_{0 \leq t \leq T} \norme{\bsym{\xi}_i^n}{\alpha}^2 \leq C_3(\alpha) \quad a.s.
\end{equation}
\begin{proof}
{Testing the noise-mode equation against $A^\alpha\bsym{\xi}_i^n$, with $1<\alpha<s$, gives}
\begin{equation}
\label{eq:estimee_energie_intermediaire_xi_s>1}
\frac12\partial_t\norme{\bsym{\xi}_i^n}{\alpha}^2+\norme{\bsym{\xi}_i^n}{s+\alpha}^2
=-b(\bsym{u}_n,\bsym{\xi}_i^n,A^\alpha\bsym{\xi}_i^n)-b(\bsym{\xi}_i^n,\bsym{u}_n,A^\alpha\bsym{\xi}_i^n).
\end{equation}
{The Sobolev embeddings used above, now with $1<\alpha<s$, imply}
\begin{equation}
\label{eq:estimee_energie_intermediaire_xi_2_s>1}
\frac12\partial_t\norme{\bsym{\xi}_i^n}{\alpha}^{2}+\norme{\bsym{\xi}_i^n}{s+\alpha}^{2}
\leq 2\norme{\bsym{\xi}_i^n}{\alpha}\norme{\bsym{u}_n}{V}\norme{\bsym{\xi}_i^n}{s+\alpha}.
\end{equation}
{Young's inequality and Gronwall's lemma therefore give
\[
\norme{\bsym{\xi}_i^n(t)}{\alpha}^2\leq \norme{\bsym{\xi}_i^n(0)}{\alpha}^2
\exp(C\int_0^T\norme{\bsym{u}_n}{V}^2dt).
\]
Summing over $i$, using \eqref{eq:estimee_energie_u}, and integrating \eqref{eq:estimee_energie_intermediaire_xi_2_s>1} proves \eqref{eq:estimee_energie_xi_s>1}.}
\end{proof}
The same arguments apply and we similarly obtain a solution, also denoted $(\bsym{u}, \bsym{\xi}^{\infty}).$ These estimates hold at the limit:
\begin{equation}
\label{eq:estimee_energie_xi_lim_s>1}
\int_0^T \sum_{i=0}^{\infty} \norme{\bsym{\xi}_i^{\infty}}{s+\alpha}^{2} dt \leq C_3(\alpha), \hspace{10pt} \rev{\sum_{i=0}^{\infty} \sup_{0 \leq t \leq T} \norme{\bsym{\xi}_i^{\infty}}{\alpha}^2} \leq C_3(\alpha) \quad a.s.
\end{equation}
Consider now $(\bsym{u}, \bsym{\xi})$ and $(\bsym{u}^{*}, \bsym{\xi}^{*})$ two couples of martingale solutions on the same filtered probability space, with the same Brownian motion of the system \eqref{eq:equation_operateurs}. Consider $\bsym{w} = \bsym{u}-\bsym{u}^*$ and $\bsym{\chi}_i = \bsym{\xi}_i - \bsym{\xi}_i^*$ for all $i \in \N.$ The system solved by $(\bsym{w}, \bsym{\chi})$ is: 
\begin{equation}
\label{eq:systeme_w_chi}
\left\{
\begin{aligned}
d_t \bsym{w} &= -\bigl( A\bsym{w} + (\bsym{u}\bcdot \nabla)\bsym{u} - (\bsym{u}^*\bcdot \nabla)\bsym{u}^*\bigr)dt  \\ 
&\Bigl(\displaystyle\sum_{i=0}^{\infty} \nabla \bcdot \Bigl(\bsym{\xi}_{i} \bsym{\xi}_{i}^{\intercal}\nabla\bsym{u} -  \bsym{\xi}_{i}^* \bsym{\xi}_{i}^{*\intercal}\nabla\bsym{u}^*\bigr)\Bigr)dt  + \displaystyle \sum_{i=0}^{\infty} \bigl((\bsym{\xi}_i\bcdot\nabla)\bsym{u} - (\bsym{\xi}_i^*\bcdot\nabla)\bsym{u}^*\bigr)d\beta_i,  \\[0.5em]
\partial_t \bsym{\chi}_i &= - A^s\bsym{\chi}_i + (\bsym{u}\bcdot\nabla)\bsym{\xi}_i + (\bsym{\xi}_i\bcdot\nabla)\bsym{u}  - (\bsym{u}^*\bcdot\nabla)\bsym{\xi}^*_i - (\bsym{\xi}^*_i\bcdot\nabla)\bsym{u}^*, \qquad i \in \N.      
\end{aligned}
\right.
\end{equation}	
Consider $g(t) = K \displaystyle \int_0^t 1 + \normeH{\bsym{u}(r)}^2 + \norme{\bsym{u}(r)}{V}^2 dr,$ where $K \in \R{}$ will be chosen below and denote $e(t) = \exp(-g(t)).$ Apply Itô formula on $(t, \bsym{x}) \mapsto e(t)\normeH{\bsym{x}}^2.$ For all $t\in [0,T],$
\begin{multline}
\label{eq:Ito_w}
    e(t)\normeH{\bsym{w}(t)}^2 = \int_0^t e(r)\scalar[\bigg]{\bsym{w}}{\Big[\G{\bsym{\xi}}{\bsym{u}} - \G{\bsym{\xi}^*}{\bsym{u}^*}\Big] d\bsym{W}_r}{H} - \frac{1}{2} \int_0^t g'(r) e(r) \normeH{\bsym{w}}^2 dr \\
    - \int_0^t e(r) \scalar[\big]{\bsym{w}}{A \bsym{w} + B(\bsym{u}, \bsym{u}) - B(\bsym{u}, \bsym{u}^*) + \F{\bsym{\xi}}{\bsym{u}} - \F{\bsym{\xi}^*}{\bsym{u}^*}}{H}  dr \\
    + \int_0^t e(r)\norme{\G{\bsym{\xi}}{\bsym{u}} - \G{\bsym{\xi}^*}{\bsym{u}^*}}{\mathcal{L}_{2}(L^2,H)}^2 dr.
\end{multline}
We now compute the (classical) derivative of $t \mapsto \frac{1}{2}e(t) \norme{\bsym{\chi}_i(t)}{\alpha}^2:$
\begin{equation*}
    \frac{d}{dt} \frac{1}{2}\scalar{e(t)\bsym{\chi}_i}{A^{\alpha}\bsym{\chi}_i}{H} = \frac{1}{2}e'(t)\norme{\bsym{\chi}_i}{\alpha}^2 + \frac{1}{2}\scalar{e(t)d_t\bsym{\chi}_i}{A^{\alpha}\bsym{\chi}_i}{H}.
\end{equation*}
We replace $d_t \bsym{\chi}_i$ by its expression in \eqref{eq:systeme_w_chi}:
\begin{multline}
\label{eq:prod_scal_chi}
    \frac{d}{dt} \frac{1}{2}\big[ e(t)\norme{\bsym{\chi}_i}{\alpha}^2 \big] + \frac{e(t)}{2}\norme{\bsym{\chi}_i}{s+\alpha}^{2} = - \frac{1}{2}g'(t)e(t)\norme{\bsym{\chi}_i}{\alpha}^2 \\ + \frac{e(t)}{2}  \big(- b(\bsym{u},\bsym{\xi}_i,A^\alpha\bsym{\chi}_i) - b(\bsym{\xi}_i,\bsym{u},A^\alpha\bsym{\chi}_i) \\ 
    + b(\bsym{u}^*,\bsym{\xi}^*_i,A^\alpha\bsym{\chi}_i) + b(\bsym{\xi}^*_i,\bsym{u}^*,A^\alpha\bsym{\chi}_i) \big).
\end{multline}
We integrate Equation \eqref{eq:prod_scal_chi} between $0$ and $t.$ We now sum over $i \in \N.$ The sum is finite thanks to \eqref{eq:estimee_energie_intermediaire_xi_2_s>1}. Next, we add it to Equation \eqref{eq:Ito_w}:
\begin{multline}
\label{eq:majoration_w_chi}
    e(t)\Big(\normeH{\bsym{w}}^2 + \frac{1}{2} \sum_{i=0}^{\infty} \norme{\bsym{\chi}_i}{\alpha}^2\Big) + \int_0^t \norme{\bsym{w}}{V}^2 + \sum_{i=0}^{\infty} \frac{1}{2}\norme{\bsym{\chi}_i}{s+\alpha}^2 dr  \\
    =  \int_0^t e(r)\scalar[\bigg]{\bsym{w}}{\Big[\G{\bsym{\xi}}{\bsym{u}} - \G{\bsym{\xi}^*}{\bsym{u}^*}\Big] d\bsym{W}_r}{H} \\
    - \frac{1}{2} \int_0^t g'(r) e(r) \normeH{\bsym{w}}^2 dr \\
     - \int_0^t e(r) \scalar[\big]{B(\bsym{u}, \bsym{u}) - B(\bsym{u}, \bsym{u}^*)}{\bsym{w}}{H} + e(r)\scalar{\F{\bsym{\xi}}{\bsym{u}} - \F{\bsym{\xi}^*}{\bsym{u}^*}}{\bsym{w}}{H}  dr \\
    + \frac{1}{2} \int_0^t e(r)\norme{\G{\bsym{\xi}}{\bsym{u}} - \G{\bsym{\xi}^*}{\bsym{u}^*}}{\mathcal{L}_{2}(L^2,H)}^2 dr - \frac{1}{2} \sum_{i=0}^{\infty}  \int_0^t g'(r)e(r)\norme{\bsym{\chi}_i}{\alpha}^2 dr \\ 
    +  \frac{1}{2} \sum_{i=0}^{\infty} \int_0^t e(r) \big(- b(\bsym{u},\bsym{\xi}_i,A^\alpha\bsym{\chi}_i) - b(\bsym{\xi}_i,\bsym{u},A^\alpha\bsym{\chi}_i) + b(\bsym{u}^*,\bsym{\xi}^*_i,A^\alpha\bsym{\chi}_i) + b(\bsym{\xi}^*_i,\bsym{u}^*,A^\alpha\bsym{\chi}_i) \big) dr.
\end{multline}
We can give an upper bound for the terms on the right hand side of Equation \eqref{eq:majoration_w_chi}. First, we treat the terms associated to the different operators. Then we will focus on the martingale term and the two correction terms namely $\frac{1}{2} \int_0^t g'(r) e(r) \normeH{\bsym{w}}^2 dr$ and $\frac{1}{2} \int_0^t g'(r)e(r)\norme{\bsym{\chi}_i}{\alpha}^2 dr.$ We start with operator $B:$ 
\begin{align*}
\scalar{B(\bsym{u},\bsym{u})}{\bsym{w}}{H} - \scalar{B(\bsym{u}^*,\bsym{u}^*)}{\bsym{w}}{H} &= \scalar{B(\bsym{w}, \bsym{u})}{\bsym{w}}{H} \\
&\leq C \normeH{\bsym{w}} \norme{\bsym{u}}{V}\norme{\bsym{w}}{V}.
\end{align*}
We now study the terms with operators $F$ and $G:$  
\begin{align*}
\Lambda &= - \sum_{i=0}^{\infty} \frac{1}{2}\scalar{\bsym{\xi}_{i} \bsym{\xi}_{i}^{\intercal}\nabla\bsym{u} -  \bsym{\xi}_{i}^* \bsym{\xi}_{i}^{*\intercal}\nabla\bsym{u}^*}{\nabla\bsym{w}}{H} + \frac{1}{2} \norme{\G{\bsym{\xi}}{\bsym{u}} - \G{\bsym{\xi}^*}{\bsym{u}^*}}{\mathcal{L}_{2}(L^2,H)}^2 \\
& = \sum_{i=0}^{\infty} - \frac{1}{2}\scalar{\bsym{\xi}_{i} \bsym{\xi}_{i}^{\intercal}\nabla\bsym{u} -  \bsym{\xi}_{i}^* \bsym{\xi}_{i}^{*\intercal}\nabla\bsym{u}^*}{\nabla\bsym{w}}{H} + \frac{1}{2} \norme{\bsym{\xi}_i^\intercal \nabla \bsym{u} - \bsym{\xi}_i^{* \intercal} \nabla \bsym{u}^*}{\mathcal{L}_{2}(L^2,H)}^2.
\end{align*}
$\Lambda$ reduces to: 
\begin{align*}
\Lambda &= \frac{1}{2}\sum_{i=0}^{\infty} \scalar[\big]{\bsym{\xi}_{i} \bsym{\xi}_{i}^{\intercal}\nabla\bsym{u}}{\nabla \bsym{u}^*}{H} + \scalar{\bsym{\xi}_{i}^* \bsym{\xi}_{i}^{*\intercal}\nabla\bsym{u}^*}{\nabla \bsym{u}}{H}  - 2 \scalar{\bsym{\xi}_i^\intercal \nabla \bsym{u}}{\bsym{\xi}_i^{* \intercal} \nabla \bsym{u}^*}{H} \\
&= \frac{1}{2}\sum_{i=0}^{\infty} \scalar{\bsym{\xi}_{i} \bsym{\xi}_{i}^{\intercal}\nabla\bsym{u} + \bsym{\xi}_{i}^* \bsym{\xi}_{i}^{*\intercal}\nabla\bsym{u} - 2 \bsym{\xi}_{i}^*\bsym{\xi}_i^\intercal \nabla \bsym{u}}{\nabla \bsym{u}^*}{H} \\
&= \Lambda_1 + \Lambda_2.
\end{align*}
Where 
\begin{align*}
\Lambda_1 &= \frac{1}{2}\sum_{i=0}^{\infty} \scalar{\bsym{\xi}_{i} \bsym{\xi}_{i}^{\intercal}\nabla\bsym{u} + \bsym{\xi}_{i}^* \bsym{\xi}_{i}^{*\intercal}\nabla\bsym{u} - 2 \bsym{\xi}_{i}^*\bsym{\xi}_i^\intercal \nabla \bsym{u}}{\nabla \bsym{u}^*-\bsym{u}}{H} \\
\Lambda_2 &= \frac{1}{2}\sum_{i=0}^{\infty} \scalar{\bsym{\xi}_{i} \bsym{\xi}_{i}^{\intercal}\nabla\bsym{u} + \bsym{\xi}_{i}^* \bsym{\xi}_{i}^{*\intercal}\nabla\bsym{u} - 2 \bsym{\xi}_{i}^*\bsym{\xi}_i^\intercal \nabla \bsym{u}}{\nabla \bsym{u}}{H}.
\end{align*}
We have to estimate both $\Lambda_1$ and $\Lambda_2.$ Start with $\Lambda_1:$
\begin{align*}
\Lambda_1 = - \frac{1}{2} \displaystyle\sum_{i=0}^{\infty} \scalar{[\bsym{\chi}_i\bsym{\xi}_i^\intercal - \bsym{\xi}_i^* \bsym{\chi}_i^\intercal] \nabla \bsym{u}}{\nabla \bsym{w}}{H}. \\
\end{align*}
Now $\Lambda_2:$
\begin{align*}
\Lambda_2 &= \frac{1}{2} \displaystyle\sum_{i=0}^{\infty}\frac{1}{2} \scalar{[\bsym{\xi}_{i} \bsym{\xi}_{i}^{\intercal} + \bsym{\xi}_{i}^* \bsym{\xi}_{i}^{*\intercal} -2 \bsym{\xi}_{i}^*\bsym{\xi}_i^\intercal] \nabla \bsym{u}}{\nabla \bsym{u}}{H} \frac{1}{2} \scalar{[\bsym{\xi}_{i} \bsym{\xi}_{i}^{\intercal} + \bsym{\xi}_{i}^* \bsym{\xi}_{i}^{*\intercal} -2 \bsym{\xi}_{i}^*\bsym{\xi}_i^\intercal] \nabla \bsym{u}}{\nabla \bsym{u}}{H}  \\
&=  \frac{1}{2}\displaystyle \sum_{i=0}^{\infty} \frac{1}{2} \scalar{[\bsym{\xi}_{i} \bsym{\xi}_{i}^{\intercal} + \bsym{\xi}_{i}^* \bsym{\xi}_{i}^{*\intercal} -2 \bsym{\xi}_{i}^*\bsym{\xi}_i^\intercal] \nabla \bsym{u}}{\nabla \bsym{u}}{H} + \frac{1}{2} \scalar{ \nabla \bsym{u}}{[\bsym{\xi}_{i} \bsym{\xi}_{i}^{\intercal} + \bsym{\xi}_{i}^* \bsym{\xi}_{i}^{*\intercal} -2 \bsym{\xi}_{i}\bsym{\xi}_i^{*\intercal}]\nabla \bsym{u}}{H} \\
&= \frac{1}{2} \displaystyle \sum_{i=0}^{\infty} \scalar{[\bsym{\xi}_{i} \bsym{\xi}_{i}^{\intercal} + \bsym{\xi}_{i}^* \bsym{\xi}_{i}^{*\intercal} - \bsym{\xi}_{i}^*\bsym{\xi}_i^{\intercal}- \bsym{\xi}_{i}\bsym{\xi}_i^{*\intercal}]\nabla \bsym{u}}{\nabla \bsym{u}}{H} \\
&= \frac{1}{2} \displaystyle \sum_{i=0}^{\infty} \scalar{\bsym{\chi}_i\bsym{\chi}_i^\intercal \nabla \bsym{u}}{\nabla\bsym{u}}{H}.
\end{align*}
Hence the equality 
\begin{align*}
\Lambda &= \frac{1}{2}\displaystyle \sum_{i=0}^{\infty} -\scalar{[\bsym{\chi}_i\bsym{\xi}_i^\intercal - \bsym{\xi}_i^* \bsym{\chi}_i^\intercal] \nabla \bsym{u}}{\nabla \bsym{w}}{H} + \scalar{\bsym{\chi}_i\bsym{\chi}_i^\intercal \nabla \bsym{u}}{\nabla\bsym{u}}{H}. \\
\end{align*}
Remark that $\scalar{\bsym{\chi}_i\bsym{\chi}_i^\intercal \nabla \bsym{u}}{\nabla\bsym{u}}{H} \geq 0,$ can withdrawn since it is on the right hand side of equation \eqref{eq:majoration_w_chi}. As for the other terms, thanks to Young inequality and estimation \eqref{eq:estimee_energie_u}, \eqref{eq:estimee_energie_xi},
\begin{align*}
    \scalar{\bsym{\chi}_i\bsym{\xi}_i^\intercal  \nabla \bsym{u}}{\nabla \bsym{w}}{H} &\leq \norme{\bsym{\chi}_i}{L^\infty} \norme{\bsym{\xi}_i}{L^\infty} \norme{\bsym{u}}{V} \norme{\bsym{w}}{V} \\
    &\leq C \norme{\bsym{\chi}_i}{\alpha}^2 \norme{\bsym{u}}{V}^2 + \epsilon \norme{\bsym{w}}{V}^2.
\end{align*}
Clearly the second term in $\Lambda$ satisfies the same estimate.

We give an upper bound for all the terms from the second equation. The right hand side reduces to:
\begin{equation*}
	- \bigl(b(\bsym{u},\bsym{\chi}_i,A^\alpha\bsym{\chi}_i) + b(\bsym{\chi}_i,\bsym{u},A^\alpha\bsym{\chi}_i) + b(\bsym{w},\bsym{\xi}_i^*,A^\alpha\bsym{\chi}_i) + b(\bsym{\xi}_i^*,\bsym{w},A^\alpha\bsym{\chi}_i)\bigr).
\end{equation*} 
We have thanks to Sobolev embedding and Young inequality, the following three upper bound:
\begin{align*}
b(\bsym{u},\bsym{\chi}_i,A^\alpha\bsym{\chi}_i) &\leq C\normeH{\bsym{u}}^2\norme{\bsym{\chi}_i}{\alpha}^2 + \epsilon \norme{\bsym{\chi}_i}{s+\alpha}^2, \\
b(\bsym{\chi}_i,\bsym{u},A^\alpha\bsym{\chi}_i) &\leq C\norme{\bsym{\chi}_i}{\alpha}^2\norme{\bsym{u}}{V}^2 + \epsilon \norme{\bsym{\chi}_i}{s+\alpha}^2, \\
b(\bsym{w},\bsym{\xi}_i^*,A^\alpha\bsym{\chi}_i) &\leq C\normeH{\bsym{w}}^2\norme{\bsym{\xi}_i^*}{V}^2 + \epsilon \norme{\bsym{\chi}_i}{s+\alpha}^2. 
\end{align*}
As for the last term, thanks to Sobolev embedding $H^s \hookrightarrow L^{\infty}.$ We apply Young inequality with parameters $\frac{1}{2}, \frac{\theta}{2}$ and $\frac{1-\theta}{2}.$ It follows that: 
\begin{align*}
b(\bsym{\xi}_i^*,\bsym{w},A^\alpha\bsym{\chi}_i) &\leq \norme{\bsym{\xi}_i^*}{L^\infty}\norme{\bsym{w}}{V} \norme{\bsym{\chi}_i}{\alpha}^{\theta}\norme{\bsym{\chi}_i}{s+\alpha}^{1-\theta} \quad \theta = \frac{s-\alpha}{s}\\
&\leq C\norme{\bsym{\xi}_i^*}{L^\infty}^{2/\theta}\norme{\bsym{\chi}_i}{\alpha}^{2} + \epsilon \norme{\bsym{\chi}_i}{s+\alpha}^{2}+ \eta \norme{\bsym{w}}{V}^2.
\end{align*}
With these upper bounds, we can write{\footnote{In the estimates above, the summation over the mode index should of course be performed before applying Young’s inequality}}: 
\begin{multline*}
    e(t)\Big(\normeH{\bsym{w}}^2 + \frac{1}{2} \sum_{i=0}^{\infty} \norme{\bsym{\chi}_i}{\alpha}^2\Big) + \int_0^t \norme{\bsym{w}}{V}^2 + \frac{1}{2}  \sum_{i=0}^{\infty} \norme{\bsym{\chi}_i}{s+\alpha}^2 dr \\ \leq  \int_0^t e(r)\scalar[\bigg]{\bsym{w}}{\Big[\G{\bsym{\xi}}{\bsym{u}} - \G{\bsym{\xi}^*}{\bsym{u}^*}\Big] d\bsym{W}_r}{H} \\
    - \frac{1}{2} \int_0^t g'(r) e(r) \Big(\normeH{\bsym{w}}^2 + \sum_{i=0}^\infty \norme{\bsym{\chi}_i}{\alpha}^2\Big) dr  \\
    +\int_0^t \Big(\normeH{\bsym{w}}^2 +  \sum_{i=0}^\infty \norme{\bsym{\chi}_i}{\alpha}^2\Big) e(r)C\big(\norme{\bsym{u}}{V}^2 + \normeH{\bsym{u}}^2 + 1\big)dr.
\end{multline*}
In the expression of function $g,$ we choose $K = 2C$ and we are left with: 
\begin{multline*}
    e(t)\Big(\normeH{\bsym{w}}^2 + \frac{1}{2} \sum_{i=0}^{\infty} \norme{\bsym{\chi}_i}{\alpha}^2\Big) + \int_0^t \norme{\bsym{w}}{V}^2 + \frac{1}{2} \sum_{i=0}^{\infty} \norme{\bsym{\chi}_i}{s+\alpha}^2 dr \\ \leq \int_0^t e(r)\scalar[\bigg]{\bsym{w}}{\Big[\G{\bsym{\xi}}{\bsym{u}} - \G{\bsym{\xi}^*}{\bsym{u}^*}\Big] d\bsym{W}_r}{H}.
\end{multline*}
We now take the mean in above equation. Observe that 
$$
\norme{\G{\bsym{\xi}}{\bsym{u}}}{\mathcal{L}_2(H)}^2=\sum_{i=0}^\infty |\xi\cdot\nabla u|^2_H \leq \sum_{i=0}^\infty \norme{\xi}{L^\infty}\norme{\nabla u}{L^2}^2.
$$
By Sobolev embedding, \eqref{eq:estimee_energie_xi_s>1}, and \eqref{eq:estimation_limite}, we deduce that this is integrable in time. This is also true with $u^*, \,\xi^*$ instead of $u,\,\xi$. It follows that
the stochastic process $\displaystyle\int_0^t e(r)\scalar[\bigg]{\bsym{w}}{\Big[\G{\bsym{\xi}}{\bsym{u}} - \G{\bsym{\xi}^*}{\bsym{u}^*}\Big] d\bsym{W}_r}{H}$ is a martingale and therefore it is of mean $0.$ Hence
\begin{equation*}
    \esp[\bigg]{e(t)\Big(\normeH{\bsym{w}}^2 + \frac{1}{2} \sum_{i=0}^{\infty} \norme{\bsym{\chi}_i}{\alpha}^2\Big) + \int_0^t \norme{\bsym{w}}{V}^2 +  \sum_{i=0}^{\infty} \norme{\bsym{\chi}_i}{s+\alpha}^2 dr} \leq 0.
\end{equation*}
We deduce that for all $t\in [0,T]$
\begin{equation*}
\esp[\bigg]{e(t)\Big(\normeH{\bsym{w}}^2 + \frac{1}{2} \sum_{i=0}^{\infty} \norme{\bsym{\chi}_i}{\alpha}^2\Big)} = 0.
\end{equation*}
And therefore, for all $t\in[0,T],$ 
\begin{equation*}
    \normeH{\bsym{w}}^2 + \frac{1}{2} \sum_{i=0}^{\infty} \norme{\bsym{\chi}_i}{\alpha}^2 =0 \quad \mathbb{P}-a.s. 
\end{equation*}
Since it is a continuous function, we conclude that 
\begin{equation*}
    \mathbb{P}-a.s. \quad \normeH{\bsym{w}}^2 + \frac{1}{2} \sum_{i=0}^{\infty} \norme{\bsym{\chi}_i}{\alpha}^2 = 0 \quad \text{for all} \quad t\in [0,T].
\end{equation*}
We can now use the argument due to \rev{Gy\"ongy and Krylov} (and we refer to \cite{Gyongy_Krilov} or to the proof of Theorem 2.2 in \cite{Mikulevicius_Rozovskii_2005} for further details) and conclude that $(\bsym{u}_n, \bsym{\xi}^n)_{n}$ converges towards a unique solution in probability in the original stochastic basis. This proves the uniqueness statement in Theorem \ref{thm:well_posedness_2d} (ii).

\section{The 3D case}\label{sec:the3d_case}
In this section we consider the case $d=3$ and prove Theorem  \ref{thm:well_posedness_3d}. We fix $\gamma > 3/2$ and start 
with energy estimates. Consider the Galerkin solutions $(\bsym{u}_n, \bsym{\xi}^n)$ for all $n \in \N$ where $\bsym{\xi}^n = (\bsym{\xi}_i^n)_{i \in \N}.$ \rev{These solutions are global and well defined (see Section \ref{sec:approximation_scheme}),} and we have the following energy estimates:
\begin{equation}
\label{eq:estimee_energie_xi_3d}
\int_0^T \sum_{i=0}^{\infty} \norme{\bsym{\xi}^n_i}{s}^{2} dt \leq C_4(s), \hspace{10pt} \sum_{i=0}^{\infty} \sup_{0 \leq t \leq T} \normeH{\bsym{\xi}_i^n}^2 \leq C_4(s) \quad a.s.
\end{equation}
\begin{proof}
{Testing the projected noise equation against $\bsym{\xi}_i^n$ gives}
\begin{equation}
\label{eq:equation_xi_3d}
\partial_t \bsym{\xi}^n_i + A^{s} \bsym{\xi}^n_i + B^n(\bsym{\xi}^n_i,\bsym{u}_n) + B^n(\bsym{u}_n,\bsym{\xi}^n_i) = 0,
\end{equation}
\begin{equation}
\label{eq:estimee_energie_intermediaire_xi_3d}
\frac12\partial_t\normeH{\bsym{\xi}_i^n}^2+\norme{\bsym{\xi}_i^n}{s}^2
=-b(\bsym{u}_n,\bsym{\xi}_i^n,\bsym{\xi}_i^n)-b(\bsym{\xi}_i^n,\bsym{u}_n,\bsym{\xi}_i^n).
\end{equation}
{Using $H^1\hookrightarrow L^6$, $H^{1/2}\hookrightarrow L^3$, $H^{3/4}\hookrightarrow L^4$, and interpolation between $H$ and $H^s$, both trilinear terms are bounded by}
\[
C\norme{\bsym{u}_n}{V}\normeH{\bsym{\xi}_i^n}^{2-\frac{3}{2s}}\norme{\bsym{\xi}_i^n}{s}^{\frac{3}{2s}}.
\]
{Young's inequality yields, for every $\varepsilon>0$,}
\[
\frac12\partial_t\normeH{\bsym{\xi}_i^n}^2+ (1-\varepsilon)\norme{\bsym{\xi}_i^n}{s}^2
\leq C_\varepsilon\norme{\bsym{u}_n}{V}^2\normeH{\bsym{\xi}_i^n}^2.
\]
{Gronwall's lemma, \eqref{eq:estimee_energie_u}, and summation over $i$ give \eqref{eq:estimee_energie_xi_3d}.}
\end{proof}
These estimates enable us to prove the following tightness result:
\begin{lemme}
\label{lemme:tightness_3d}
In dimension $d=3,$ the following statements hold:
\begin{enumerate}[label=(\roman*)]
    \item The sequence $\bigl(\mathcal{L}(\bsym{u}_{n})\bigr)_{n \in \N}$ is tight in $L^{p}(0,T,H), \quad p >1.$

    \item The sequence $\bigl(\mathcal{L}(\bsym{\xi}_{i}^n)_{n \in \N}\bigr)_{i \in \N}$ is tight in $\bigl(L^{2}(0,T,H)\bigr)^\N$ 

    \item The sequence $\bigl(\mathcal{L}(\bsym{u}_{n})\bigr)_{n \in \N}$ is tight in $C^0\bigl(0,T,\domA{-\gamma}\bigr)$

    \item The sequence $\bigl(\mathcal{L}(\bsym{\xi}_i^n)_{n \in \N}\bigr)_{i \in \N}$ is tight in $\bigl(C^0\big(0,T,\domA{\frac{q}{2}}\big)\bigr)^\N, \quad q<0.$ 
    
\end{enumerate}    
\end{lemme}
\rev{The proof is similar to the case $d=2$.} We use similar arguments,  merely adjusting the Sobolev embeddings for higher dimensions. First, let us rewrite the system:
\begin{equation}
\left\{
\begin{aligned}
\bsym{u}_{n}(t) = J_{1}^n + J_{2}^n + J_{3}^n + J_{4}^n + J_{5}^n, \\[0.5em]
\bsym{\xi}^n_i(t) = K_{1}^n + K_{2}^n + K_{3}^n + K_{4}^n.
\end{aligned}
\right.
\end{equation}
\begin{proof}
\label{proof:tightness_3d}
\begin{enumerate}[label=(\roman*)]
\item Fix $p >1$ to be chosen below and $\delta = \frac{1}{2}.$ As in the case $d=2,$ we apply the following compact embedding result:
\begin{equation*}
L^p(0,T,V) \cap W^{\delta, 2}\bigl(0,T,\domA{-\gamma}\bigr) \hookrightarrow L^{p}(0,T,H).
\end{equation*}
Energy estimate \eqref{eq:estimee_energie_u} shows that the laws $\bigl(\mathcal{L}(\bsym{u}_{n})\bigr)_{n \in \N}$ are bounded in probability in $L^2(0,T,V).$ We now prove that they are bounded in probability in $W^{\alpha, p}\bigl(0,T,\domA{-\gamma}\bigr).$ The proof of Lemma \ref{lemme:tightness_2d} adapts for the following terms and we have: 
\begin{align*}
\normeH{J_{1}^n}^2 &\leq K \quad a.s. \\
\norme{J_2^n}{W^{1, 2}(0,T,V')}^{2} &\leq K \quad a.s. \\
 \norme{J_3^n}{W^{1, 2}\bigl(0,T,\domA{-\gamma}\bigr)}^2 &\leq K \quad a.s. \\
 \norme{J_5^n}{W^{\delta, q}\bigl(0,T,\domA{-\gamma}\bigr)}^2 &\leq K \quad a.s. \quad q>2.
\end{align*}
For $J_4^n,$ thanks to \eqref{eq:F_majoration}, we have:
\begin{align*}
\norme{J_4^n}{W^{1, p}\bigl(0,T,\domA{-\gamma}\bigr)}^2 &\leq \int_0^T \norme{\F[n]{\bsym{\xi}^n}{\bsym{u}_n}}{\domA{-\gamma}}^p dr \\
&\leq \int_0^T \Big(\sum_{i=0}^{\infty}  \norme{\bsym{\xi}_i^n}{L^4}^2 \norme{\bsym{u}_n}{V}\Big)^p dr .
\end{align*}
When $d=3,$ we have $H^{3/4} \hookrightarrow L^4.$ By interpolation, since $0 < 3/4 < s,$ we have:
\begin{equation*}
\norme{\bsym{\xi}_i^n}{3/4} \leq \normeH{\bsym{\xi}_i^n}^{\frac{4s-3}{4s}} \norme{\bsym{\xi}_i^n}{s}^{\frac{3}{4s}}.    
\end{equation*}
Thus:
\begin{equation*}
\norme{J_4^n}{W^{1, p}\bigl(0,T,\domA{-\gamma}\bigr)}^2 \leq  \int_0^T \norme{\bsym{u}_n}{V}^p \Big(\sum_{i=0}^{\infty}  \normeH{\bsym{\xi}_i^n}^{\frac{4s-3}{2s}} \norme{\bsym{\xi}_i^n}{s}^{\frac{3}{2s}}\Big)^p  dr.  
\end{equation*}
\rev{By Hölder's inequality, with} $\frac{1}{q} + \frac{1}{\tilde{q}} = 1,$ and estimate \eqref{eq:estimee_energie_xi_3d}, 
\begin{multline*}
\norme{J_4^n}{W^{1, p}\bigl(0,T,\domA{-\gamma}\bigr)}^2 \leq C  \Big(\int_0^T \norme{\bsym{u}_n}{V}^{q p} dr\Big)^{1/q} \\\Big(\int_0^T \Big(\sum_{i=0}^{\infty} \normeH{\bsym{\xi}_i^n}^{\frac{4s-3}{2s}} \norme{\bsym{\xi}_i^n}{s}^{\frac{3}{2s}}\Big)^{\tilde{q} p} dr\Big)^{1/\tilde{q}} .  
\end{multline*}
\rev{Applying Hölder's inequality once more to the sum, with} $\frac{1}{q_1} + \frac{1}{q_2} = 1,$ we have: 
\begin{multline*}
\norme{J_4^n}{W^{1, p}\bigl(0,T,\domA{-\gamma}\bigr)}^2 \leq C \Big(\int_0^T \norme{\bsym{u}_n}{V}^{q p} dr\Big)^{1/q} \\
\Big(\int_0^T \Big(\sum_{i=0}^{\infty} \normeH{\bsym{\xi}_i^n}^{q_1\frac{4s-3}{2s}}\Big)^{\frac{p\tilde{q}}{q_1}} \Big(\sum_{i=0}^{\infty} \norme{\bsym{\xi}_i^n}{s}^{q_2\frac{3}{2s}} \Big)^{\frac{p\tilde{q}}{q_2}}dr\Big)^{1/\tilde{q}}. 
\end{multline*}
We choose $q = \frac{2}{p}$ such that $\tilde{q} = \frac{2}{2-p}.$ Thanks to \eqref{eq:estimee_energie_u} and \eqref{eq:estimee_energie_xi_3d}, we are left with
\begin{equation*}
\norme{J_4^n}{W^{1, p}\bigl(0,T,\domA{-\gamma}\bigr)}^2 \leq C \Big(\int_0^T \Big(\sum_{i=0}^{\infty} \norme{\bsym{\xi}_i^n}{s}^{q_2\frac{3}{2s}} \Big)^{\frac{2p}{(2-p)q_2}}dr\Big)^{1/\tilde{q}}.
\end{equation*}
We take $q_2 = \frac{2p}{2-p}.$ Hence
\begin{equation*}
\norme{J_4^n}{W^{1, p}\bigl(0,T,\domA{-\gamma}\bigr)}^2 \leq C \Big(\int_0^T \sum_{i=0}^{\infty} \norme{\bsym{\xi}_i^n}{s}^{\frac{2p}{2-p}\frac{3}{2s}} dr\Big)^{1/\tilde{q}}.
\end{equation*}
We can conclude thanks to estimate \eqref{eq:estimee_energie_xi_3d} that 
\begin{equation*}
\norme{J_4^n}{W^{1, p}\bigl(0,T,\domA{-\gamma}\bigr)}^2 \leq K  \quad a.s. 
\end{equation*}
whenever $\frac{3p}{s(2-p)} \leq 2.$ Since $s > 3/2,$ we can find $p \in (1,2)$ so that this holds. Therefore $\mathcal{L}(\bsym{u}_n)$ is tight in $L^p(0,T,H).$
\item Similar to the case $d=2.$
\item Let $p > 1$ and $ap >1.$ We apply as in the case $d= 2$ the compact embedding 
\begin{equation*}
W^{a,p}\bigl(0,T,\domA{-\gamma}\bigr) \hookrightarrow C^0\bigl(0,T,\domA{-\beta} \bigr).
\end{equation*}
\rev{All the upper bounds are obtained as in the case $d=2$, except for $J_4^n$.} However, we just proved that
\begin{equation*}
\norme{J_4^n}{W^{1, p}\bigl(0,T,\domA{-\gamma}\bigr)}^2 \leq K \quad a.s.
\end{equation*}
As $p > 1,$ we can conclude that $\big(\mathcal{L}(\bsym{u}_n)\big)_{n \in \N}$ is tight in $C^0\big(0,T,\domA{-\gamma}\big).$
\item Similar to the case $d=2.$
\end{enumerate} 
\end{proof}
For all $q <0,$ for all $ p > 1,$ the family of laws of the triplet $(\bsym{u}_n, \bsym{\xi}^n, \bsym{W}^n)$ is tight in 
\begin{equation*}
\bigl(L^{p}(0,T,H) \cap C^0\bigl(0,T,\domA{-\gamma}\bigr)\bigr) \times \bigl(L^2(0,T,H) \cap C^0\big(0,T,\domA{\frac{q}{2}}\big)\bigr)^{\N} \times C^0\bigl(0,T,\domA{-\gamma}\bigr).
\end{equation*}
\rev{By the Skorokhod representation theorem, there exist a stochastic basis} $\bigl(\bar{\Omega}, \bar{\mathcal{F}}, \{\bar{\mathcal{F}}_t\}_{t \geq 0}, \bar{\mathbb{P}}, \bar{\bsym{W}}^{n}\bigr),$ a sequence of random variables taking values in the above space $(\bar{\bsym{u}}_n, \bar{\bsym{\xi}}^n, \bsym{W}^n),$ and a limit triple $(\bar{\bsym{u}}, \bar{\bsym{\xi}}^{\infty}, \bar{\bsym{W}})$ such that $(\bar{\bsym{u}}_n, \bar{\bsym{\xi}}^n, \bar{\bsym{W}}^n)$ has the same law as $(\bsym{u}_n, \bsym{\xi}^n, \bsym{W}^n).$ \rev{Up to a subsequence, these sequences converge}
\begin{equation}
\label{eq:convergence_3d}
\begin{aligned}
\bar{\bsym{u}}_n &\longrightarrow \bar{\bsym{u}} \quad \text{in} \quad L^{p}(0,T,H) \cap C^0\bigl(0,T,\domA{-\gamma}\bigr)  \quad \bar{\mathbb{P}}-a.s. \\
\bar{\bsym{\xi}}^n &\longrightarrow \bar{\bsym{\xi}}^{\infty} \quad \text{in} \quad \bigl(L^2(0,T,H) \cap C^0\big(0,T,\domA{\frac{q}{2}}\big)\bigr)^{\N} \quad \bar{\mathbb{P}}-a.s. \\
\bar{\bsym{W}}^n &\longrightarrow \bar{\bsym{W}} \quad \text{in} \quad C^0\bigl(0,T,\domA{-\gamma}\bigr) \quad \bar{\mathbb{P}}-a.s.
\end{aligned}
\end{equation} 
\rev{This provides a candidate solution. It remains to verify that it satisfies the equation.} Consider again the variational problem:
\begin{equation}
\left\{
	\begin{aligned}
\scalar{\bar{\bsym{u}}_{n}(t) - \bar{\bsym{u}}_{n}(0)}{\bsym{y}}{H} + \displaystyle\int_{0}^t \scalar{A\bar{\bsym{u}}_n}{\bsym{y}}{H} + \scalar{B^n(\bar{\bsym{u}}_n,\bar{\bsym{u}}_n)}{\bsym{y}}{H} \\
+ \scalar{\F[n]{\bar{\bsym{\xi}}^n}{\bar{\bsym{u}}_n}}{\bsym{y}}{H} dr = \scalar[\bigg]{\displaystyle\int_0^t \G[n]{\bar{\bsym{\xi}}^n}{\bar{\bsym{u}}_n}d\bar{\bsym{W}}^n_r}{\bsym{y}}{H}, \\[0.5em]
\scalar{\bar{\bsym{\xi}}^n_i(t) - \bar{\bsym{\xi}}^n_i(0)}{\bsym{z}}{H} + \displaystyle \int_{0}^t \scalar{A^s \bar{\bsym{\xi}}^n_i}{\bsym{z}}{H} + \scalar{B^n(\bar{\bsym{\xi}}^n_i,\bar{\bsym{u}}_n)}{\bsym{z}}{H}  \\
+ \scalar{B^n(\bar{\bsym{u}}_n,\bar{\bsym{\xi}}^n_i)}{\bsym{z}}{H} dr = 0, \qquad i \in \N.
	\end{aligned}
	\right.
\end{equation}
\begin{lemme}
\label{lemme:convergence_3d}
Consider the stochastic basis $\bigl(\bar{\Omega}, \bar{\mathcal{F}}, \{\bar{\mathcal{F}}_t\}_{t \geq 0}, \bar{\mathbb{P}}, \bar{\bsym{W}}^{n}\bigr).$ \rev{Up to a subsequence, the following almost sure convergences hold:} 

\begin{enumerate}[label=(\roman*)]
    \item $J_i^n \longrightarrow J_i \quad \bar{\mathbb{P}}-a.s. \quad i \in \{1,..,5\}.$
    \item $K_i^n \longrightarrow K_i \quad \bar{\mathbb{P}}-a.s. \quad i \in \{1,..,4\}.$
\end{enumerate}
\end{lemme}
\begin{proof}
The proof that each term converges to its limit is exactly the same in the case $d=2$ except for  $\scalar{\F[n]{\bar{\bsym{\xi}}^n}{\bar{\bsym{u}}_n}}{\bsym{y}}{H}.$ We write:
\begin{equation*}
\begin{aligned}
\int_0^t \scalar{\F[n]{\bar{\bsym{\xi}}^n}{\bar{\bsym{u}}_n}}{\bsym{y}}{H} - \scalar{\F{\bar{\bsym{\xi}}^{\infty}}{\bar{\bsym{u}}}}{\bsym{y}}{H} dr = \int_0^t \scalar{\F{\bar{\bsym{\xi}}^n}{\bar{\bsym{u}}_n}}{P_n\bsym{y}-\bsym{y}}{H} \\
+ \scalar{\F{\bar{\bsym{\xi}}^n}{\bar{\bsym{u}}_n} - \F{\bar{\bsym{\xi}}^{\infty}}{\bar{\bsym{u}}_n}}{\bsym{y}}{H} \\
+ \scalar{\F{\bar{\bsym{\xi}}^{\infty}}{\bar{\bsym{u}}_n}- \F{\bar{\bsym{\xi}}^{\infty}}{\bar{\bsym{u}}}}{\bsym{y}}{H} dr \\
= J_{4,1}^n + J_{4,2}^n + J_{4,3}^n.
\end{aligned} 
\end{equation*}
Both $J_{4,1}^n$ and $J_{4,3}^n$ can be treated as in the case $d=2.$ As for $J_{4,2}^n,$ we have:
\begin{equation*}
\begin{aligned}
\int_0^t \scalar{\F{\bar{\bsym{\xi}}^n}{\bar{\bsym{u}}_n} - \F{\bar{\bsym{\xi}}^{\infty}}{\bar{\bsym{u}}_n}}{\bsym{y}}{H} dr \leq \norme{\nabla \bsym{y}}{L^{\infty}} \\\int_0^t \norme{\bar{\bsym{u}}_n}{V} \normeH{\sum_{i=0}^{\infty}\bigl(\bar{\bsym{\xi}}_i^n\bar{\bsym{\xi}}_i^{n\intercal} - \bar{\bsym{\xi}}_i^{\infty}\bar{\bsym{\xi}}_i^{\infty\intercal}\bigr)} dr. 
\end{aligned}
\end{equation*}
Under the integral: 
\begin{align*}
\sum_{i=0}^{\infty} \normeH{\bigl(\bar{\bsym{\xi}}_i^n\bar{\bsym{\xi}}_i^{n\intercal} - \bar{\bsym{\xi}}_i^{\infty}\bar{\bsym{\xi}}_i^{\infty\intercal}\bigr)} ds &=  \sum_{i=0}^{\infty} \normeH{\bar{\bsym{\xi}}_i^n \bigl(\bar{\bsym{\xi}}_i^{n\intercal}-\bar{\bsym{\xi}}_i^{\infty\intercal}\bigr) + \bigl(\bar{\bsym{\xi}}_i^n-\bar{\bsym{\xi}}_i^{\infty}\bigr)\bar{\bsym{\xi}}_i^{\infty\intercal}} \\
&\leq 2 \sum_{i=0}^{\infty} \norme{\bar{\bsym{\xi}}_i^n}{L^4} \norme{\bar{\bsym{\xi}}_i^n - \bar{\bsym{\xi}}_i^{\infty}}{L^4}.
\end{align*}
By Sobolev embedding, $H^{3/4} \hookrightarrow L^4.$ Hence
\begin{equation*}
\sum_{i=0}^{\infty} \norme{\bar{\bsym{\xi}}_i^n}{L^4} \norme{\bar{\bsym{\xi}}_i^n - \bar{\bsym{\xi}}_i^{\infty}}{L^4} \leq \sum_{i=0}^{\infty} \norme{\bar{\bsym{\xi}}_i^n}{3/4} \norme{\bar{\bsym{\xi}}_i^n - \bar{\bsym{\xi}}_i^{\infty}}{3/4}.
\end{equation*}
We apply Cauchy-Schwarz inequality and interpolate $0< 3/4 < s:$
\begin{equation*}
\sum_{i=0}^{\infty} \norme{\bar{\bsym{\xi}}_i^n}{L^4} \norme{\bar{\bsym{\xi}}_i^n - \bar{\bsym{\xi}}_i^{\infty}}{L^4} \leq \sum_{i=0}^{\infty} \normeH{\bar{\bsym{\xi}}_i^n}^{\frac{4s-3}{2s}} \norme{\bar{\bsym{\xi}}_i^n}{s}^{\frac{3}{2s}} \sum_{i=0}^{\infty} \norme{\bar{\bsym{\xi}}_i^n - \bar{\bsym{\xi}}_i^{\infty}}{3/4}^{2}.
\end{equation*}
Applying now Cauchy-Schwarz inequality on the integral, we obtain: 
\begin{multline*}
\int_0^t \sum_{i=0}^{\infty} \norme{\bar{\bsym{\xi}}_i^n}{L^4} \norme{\bar{\bsym{\xi}}_i^n - \bar{\bsym{\xi}}_i^{\infty}}{L^4} dr \leq \int_0^t \Bigl(\sum_{i=0}^{\infty} \normeH{\bar{\bsym{\xi}}_i^n}^{\frac{4s-3}{2s}} \norme{\bar{\bsym{\xi}}_i^n}{s}^{\frac{3}{2s}}\Bigr)^2 dr \\
\int_0^t \Bigl(\sum_{i=0}^{\infty} \norme{\bar{\bsym{\xi}}_i^n - \bar{\bsym{\xi}}_i^{\infty}}{3/4}^{2}\Bigr)^2 dr.
\end{multline*}
Thanks to \eqref{eq:estimee_energie_xi_3d}, the first integral is bounded. The second integral goes to zero by dominated convergence, indeed for each $i$ we have convergence in $L^2(O,T;\mathcal D(A^\zeta)$ for any $\zeta<s$ by interpolation.  All the other convergences are identical as the ones obtained in the proof of the case $d=2$.
\end{proof}
Therefore the couple $(\bsym{u}, \bsym{\xi}^{\infty})$ satisfies the variational problem with $s > 3/2$ for all $\bsym{y}, \bsym{z} \in \domA{\gamma}, \gamma > 3/2$ and almost surely in $(t, \omega) \in [0,T] \times \bar{\Omega}:$
\begin{equation}
\label{eq:pb_variationel_solution_3d}
\left\{
	\begin{aligned}
\scalar{\bar{\bsym{u}}(t) - \bar{\bsym{u}}(0)}{\bsym{y}}{H} + \displaystyle\int_{0}^t \scalar{A\bar{\bsym{u}}}{\bsym{y}}{H} + \scalar{B(\bar{\bsym{u}},\bar{\bsym{u}})}{\bsym{y}}{H} \\
+ \scalar{\F{\bar{\bsym{\xi}}^{\infty}}{\bar{\bsym{u}}}}{\bsym{y}}{H} ds = \scalar[\bigg]{\displaystyle\int_0^t \G{\bar{\bsym{\xi}}^{\infty}}{\bar{\bsym{u}}}d\bar{\bsym{W}}_s}{\bsym{y}}{H}, \\[0.5em]
\scalar{\bar{\bsym{\xi}}_i^{\infty}(t) - \bar{\bsym{\xi}}_i^{\infty}(0)}{\bsym{z}}{H} + \displaystyle \int_{0}^t \scalar{A^s \bar{\bsym{\xi}}_i^{\infty}}{\bsym{z}}{H} + \scalar{B(\bar{\bsym{\xi}}_i^{\infty},\bar{\bsym{u}})}{\bsym{z}}{H} \\
+ \scalar{B(\bar{\bsym{u}},\bar{\bsym{\xi}}_i^{\infty})}{\bsym{z}}{H} ds = 0, \qquad i \in \N.
	\end{aligned}
	\right.
\end{equation}
This concludes the existence result in the case $d=3$ and completes the proof of Theorem \ref{thm:well_posedness_3d}.

\paragraph*{Disclosure statement:}
On behalf of all authors, the corresponding author states that there is no conflict of interest.

\paragraph*{Acknowledgements:}
The authors acknowledge the support of the ERC EU project 856408-STUOD.

\newpage

\bibliographystyle{plain}
\bibliography{biblio}

\end{document}